\documentclass[11pt]{amsart}

\usepackage{amsthm,amssymb,amsmath,amstext,amsfonts}
\usepackage{enumitem,mathtools,pgfplots,pgfmath}

\usetikzlibrary{arrows,shapes,automata,backgrounds,decorations,petri,positioning}
\usetikzlibrary{decorations.pathreplacing,angles,quotes}

\usepackage{algorithm}
\usepackage{algpseudocode}

\usepackage[utf8]{inputenc}
\usepackage[english]{babel}
\usepackage{amsthm}
\usepackage{amsmath}
\usepackage{mathtools}
\allowdisplaybreaks
\usepackage{hyperref}
\usepackage{amssymb}
\usepackage{booktabs}
\usepackage{enumitem}
\usepackage{graphicx}
\usepackage{subcaption}
\usepackage{centernot}
\usepackage{tikz}
\usepackage{amsaddr}
\theoremstyle{plain}

\newtheorem{theorem}{Theorem}[section]
\newtheorem{corollary}[theorem]{Corollary}

\newtheorem{conjecture}[theorem]{Conjecture}

\newtheorem{lemma}[theorem]{Lemma}

\calclayout

\title{On permutation patterns with constrained gap sizes}
\author{Stoyan Dimitrov}
\address{University of Illinois at Chicago}
\email{sdimit6@uic.edu}
\date{}

\begin{document}

\maketitle
\begin{abstract}
    We consider avoidance of permutation patterns with designated gap sizes between pairs of consecutive letters. We call the patterns having such constraints \emph{distant patterns (DPs)} and we show their relation to other pattern notions investigated in the past. New results on DPs with $2$ and $3$ letters are obtained. Furthermore, we show how one can use DPs to prove two former conjectures of Kuszmaul without a computer. In addition, we deduce a surprising relation between the sets of permutations avoiding the classical patterns $123$ and $132$ by looking at a class of DPs with tight gap constraints. Some interesting analogues of the Stanley--Wilf former conjecture for DPs are also discussed.
\end{abstract}
\section{Introduction}
The notion of patterns in permutations has applications when solving a variety of enumeration problems in different areas including algorithms (sortable permutations), algebraic geometry (Kazhdan-Lusztig polynomials and Shubert varieties), statistical mechanics (generalizations of gas models), computational biology and even chemistry.

The work of Simion and Schmidt \cite{simion} was the first systematic study of permutation pattern avoidance. Before that, in 1968, Knuth \cite{knuth} showed that the permutation $\pi$ can be sorted by a stack if and only if $\pi$ avoids $231$, and that the stack-sortable permutations are enumerated by the Catalan numbers.

This work is related to the idea of considering permutation patterns for which we have additional constraints for the distance between some pairs of consecutive letters in the pattern's occurrences. For example, when we want two consecutive letters to have a positive gap size in any occurrence of a pattern, we will write $\square$ between these two letters. For instance, a permutation $\pi = \pi_{1}\pi_{2} \cdots \pi_{n}$ avoids $1 \square 2$ if there is no $1\leq i<j\leq n$ with $j\geq i+2$, such that $\pi_{i}<\pi_{j}$. In this paper, we obtain some interesting results related to this type of patterns, and show its usefulness when interpreting other results combinatorially.

\subsection{Definitions}
Permutations in this paper are presented in one-line notation. When we say an $n$-permutation or a permutation of size $n$, we will mean a bijective map from $[n] = \{1,2, \cdots , n\}$ to itself. A sequence of distinct numbers will be just called \emph{sequence}. \emph{An occurrence} of the classical pattern $p$ in a permutation $\pi$ is a subsequence in $\pi$ whose letters are in the same relative order as those in $p$. We use the $\textsf{fl}$ operator to define this formally. Given a sequence of distinct real numbers $u_{1}, u_{2}, \cdots u_{k}$, define $\textsf{fl}(u_{1}, u_{2}, \cdots u_{k})$ to be the permutation $q\in S_{k}$ such that $u_{i}<u_{j}$ if and only if $q_{i}<q_{j}$. A permutation is uniquely defined by the set of its inversions, so this condition uniquely defines $q$. A permutation $\pi\in S_n$ contains a pattern $p = p_{1}p_{2} \cdots p_{k}\in S_{k}$ for $k \leq n$ if there exists $i_{1} < i_{2} < \cdots < i_{k}$ such that $\textsf{fl}(\pi_{i_{1}}\pi_{i_{2}}\cdots \pi_{i_{k}}) = p$. Otherwise, $\pi$
avoids $p$. For instance, the permutation $32514$ has two occurrences of the pattern $231$, namely the subsequences $351$ and $251$ and it avoids the pattern $123$.

\emph{Vincular patterns} allow the requirement that some adjacent letters in a pattern must also be adjacent in the permutation.
We indicate this requirement by underlining the letters that must be adjacent. For instance, if the pattern 3\underline{12} occurs in a permutation $\pi$, then the letters in $\pi$ that correspond to $1$ and $2$ are adjacent. For example, the permutation $621543$ has
only one occurrence of the pattern $3\underline{12}$, namely the subsequence $615$. Vincular patterns were first studied systematically by Babson and Steingr\'{ı}msson in \cite{babson}, under the name “generalized patterns”. 
When all the entries of a pattern are required to occur in adjacent positions, we will call the pattern \emph{consecutive}.

We will denote by $S_{n}$ the set of permutations of size $n$ and by $\operatorname{Av}_{n}(p)$ the set of permutations of size $n$ avoiding a pattern $p$. The set of all permutations, of any size, avoiding a pattern $p$ will be denoted $\operatorname{Av}(p)$. Similarly, the set of permutations that avoid all permutations in a set $\Pi$ will be denoted $\operatorname{Av}_{n}(\Pi)$ and $\operatorname{Av}(\Pi)$, respectively. If each number in a sequence $\sigma_1$ is greater than each number in a sequence $\sigma_2$, then we write $\sigma_1 > \sigma_2$. Otherwise, we write $\sigma_1 \ngtr \sigma_2$. Finally, we write $\sigma[<x]$ (resp. $\sigma[>x]$) for the subsequence of elements of $\sigma$ that are less than (resp., greater than) $x$.


A permutation class is a set $C$ of permutations such that every pattern occurring in a permutation in $C$ is also in $C$. Certainly, $\operatorname{Av}(\Pi)$ is a permutation class for any set of permutations $\Pi$. Two permutations $\sigma , \tau\in S_{k}$  are \emph{Wilf-equivalent} if, for each $n$, $|\operatorname{Av}_{n}(\sigma)|$ equals $|\operatorname{Av}_{n}(\tau)|$. Whenever we say generating function, we will mean ordinary generating function.

We will write $\square^{r}$ to denote a gap with at least $r$ letters, with $\square \coloneqq \square^{1}$. Therefore avoiding the pattern $12 \square^{2} 3$ will be the same as avoiding occurrences $xyz$ of the classical pattern $123$ for which there are at least two other letters between $y$ and $z$. The patterns containing $\square^{r}$ symbols will be called \emph{distant patterns} (\emph{DPs}) and we will use  \emph{gap} and  \emph{gap size} for the space between two consecutive letters of a pattern and for its size. Any distant pattern can be written in the form
\begin{equation*}
    \square^{r_{0}}q_{1}\square^{r_{1}}q_{2}\square^{r_{2}} \cdots \square^{r_{k-1}}q_{k}\square^{r_{k}},
\end{equation*} where each $r_{i}$ is a non-negative integer and $q_{1}q_{2}\cdots q_{k}\in S_{k}$. We will also consider tight constraints and we will underline the corresponding part of the pattern in case of a tight constraint as, for example, in $\underline{1\square^{4} 2}3$ to denote that we want to avoid the pattern $123$ with gap size exactly $4$ between the letters $1$ and $2$. DPs without any tight constraints are \emph{classical distant patterns} while DPs having at least one tight constraint are \emph{vincular distant patterns}. If we take a classical pattern $q$ and require the minimal gap size to be the same number $r$ for all pairs of consecutive letters, we will write $\operatorname{dist}_{r}(q)$ and we will call these \emph{uniform distant patterns}. For example, $\operatorname{dist}_{3}(312)=3\square^{3}1\square^{3}2$. Note that DPs generalize classical patterns since $q = \operatorname{dist}_{0}(q)$ for any classical pattern $q$. DPs generalize vincular patterns, as well, since one can write any vincular pattern as a vincular distant pattern. Finally, when we say that a distant pattern has size $n$, we mean that the number of its non-square letters is $n$. For example $2\square 1\square^{2} 3$ is a distant pattern of size $3$.

\subsection{Related work and motivation}
The idea of arbitrary constraints for the gap sizes between any two consecutive pattern letters is not new, even though not much has been written on the subject and the topic seems to be not so well explored yet. The thesis of Ghassan Firro \cite{thesis} defines a more general concept of permutation patterns with gap constraints unifying many popular pattern notations. He calls them \emph{distanced patterns} or \emph{d-patterns} and uses a different notation. The distanced patterns described there also allow requiring a gap size to be at most some given number $r$. The thesis itself enumerates the patterns of the kind $xy \square z$ using both a direct bijection and an analytical approach. We have included this result in Section \ref{sec:polygons}. The paper of Hopkins and Weiler \cite{hopkins} describes the concept of  uniform distant patterns  under the name of \emph{gap patterns} and obtains an important result related to them, as a corollary of their work on pattern avoidance over posets. We state this corollary in Section \ref{sec:len3}.

In his book dedicated to pattern avoidance \cite{Kit11}, Kitaev mentions patterns containing the $\square$-symbol, where the work of Hou and Mansour on the so-called Horse permutations \cite{horsePerm} is listed. There, the authors proved that the permutations avoiding both the classical pattern $132$ and the pattern $1\square\underline{23}$ are in one to one correspondence with the so-called Horse paths.

In \cite{Kit05}, Kitaev introduces \emph{partially ordered patterns} (\emph{POPs}) and \emph{partially ordered generalized patterns} (\emph{POGPs}) which further generalize classical DPs (resp., vincular DPs). While in classical patterns, all of the letters form one totally ordered set (e.g. in $123$, $1<2<3$), in POPs this set is partially ordered. In an occurrence of a distant pattern, an element at the place of a $\square$ is incomparable to any other element which shows us that POPs (and POGPs that allow tight constraints) are indeed generalizations. If we have a classical distant pattern or a vincular distant pattern, we could easily write it as a POP (resp., POGP) by replacing each square with a letter in its own group. POPs were studied in the context of permutations, words and compositions in a series of papers \cite{Kit06,Kit03,Kit05Main,Kit05,Kit07,Kit10,Kit11} including a recent work \cite{Kit19} of Gao and Kitaev where a systematic search of connections between sequences in the Online Encyclopedia of Integer Sequences (OEIS) and permutations avoiding POPs of size $4$ and $5$ was conducted. Two other works \cite{callan} and \cite{claessonNon} study avoidance of non-consecutive occurrence of a pattern and this has connections with both POPs and DPs. Another generalization of the DPs are the so-called \emph{place-difference-value patterns} \cite{placeDiff}. 
\subsection{Organization of the results}
\textbf{Section \ref{sec:basic}} contains proofs of two basic facts about DPs which set the stage for the later work.

\textbf{Section \ref{sec:len2}} considers the avoidance of classical DPs of size 2, namely $2 \square 1$ and the more general case $2 \square^{r} 1$. While $|\operatorname{Av}_{n}(2 \square 1)| = |\operatorname{Av}_{n}(1 \square 2)|$ was shown to be given by the Fibonacci numbers $F_{n+1}$ in many ways in the past, we use this old result and a technique that can be generally applied to any classical DP to obtain a new summation formula for $F_{n+1}$ (see Section \ref{sec:rec}). In addition, we establish a bijection between the permutations in $\operatorname{Av}_{n}(2 \square^{r} 1)$ and the permutations in $S_{n}$ for which any two elements in a cycle differ by at most $r$. 

\textbf{Section \ref{sec:len3}} considers the DPs of size $3$. Previous results by Firro and Mansour\cite{thesis, MansourFuncEq}, as well as by Hopkins and Weiler\cite{hopkins} are first listed. Then, we describe briefly our approach towards the enumeration of $\operatorname{Av}_{n}(1\square 3 \square 2)$. The main tool that was used was the block-decomposition approach initiated by Mansour and Vainshtein \cite{MansourVain}. All the details in the proposed approach will be described in a forthcoming article.

In \textbf{Section \ref{sec:vinc}}, we classify the patterns of the form $\underline{ab}\square c$ and $a\square\underline{bc}$. Recursive formulas for $|\operatorname{Av}_{n}(1\square\underline{23})|$ and $|\operatorname{Av}_{n}(1\square\underline{32})|$ are obtained, which help us to show that $1\square\underline{23}$ is the only pattern among $1\square\underline{23}$, $\underline{12}\square 3$, $1\square 2 \square 3$ and $\underline{123}$ that is avoided by fewer permutations of size $n$, compared to the same pattern after we switch the places of $2$ and $3$. This is somewhat surprising since, as we know, $|\operatorname{Av}_{n}(123)| = |\operatorname{Av}_{n}(132)|$ (see \cite{simion}). We study consecutive DPs in \textbf{Section \ref{sec:consecutiveDP}}. We will show a simple but surprising relation between these patterns and the question of avoidance of arithmetic progressions in permutations for which still not much is known.

In \textbf{Section \ref{sec:conj}}, we present direct solutions, without using a computer, for two former conjectures of Kuzsmaul \cite{Kuszmaul} which give the generating functions for the permutations avoiding two big sets of size four patterns. These solutions use formulations of the conjectures in terms of distant patterns. 

\textbf{Section \ref{sec:SW}} is dedicated to some analogues of the Stanley--Wilf former conjecture for distant patterns. We conclude with \textbf{Section \ref{sec:open}} which lists some open problems.

\section{Two basic facts about distant patterns} \label{sec:basic}
Avoidance of classical distant patterns can be formulated as a statement about simultaneous avoidance of classical patterns. For example, avoiding $1 \square 2$ is equivalent to the simultaneous avoidance of the 3-letter classical patterns $\{123,132,213\}$. In the general case, we have the fact below, where \begin{center}
$x^{(y)} \coloneqq x(x-1)\cdots (x-y+1) = \frac{x!}{(x-y)!}$
\end{center}
denotes the falling factorial.
\begin{theorem}\label{th:simultAv}
The avoidance of $q = \square^{r_{0}}q_{1}\square^{r_{1}}q_{2}\square^{r_{2}} \cdots \square^{r_{k-1}}q_{k}\square^{r_{k}}$, where $\sum\limits_{j=0}^{k}r_{j} = S$, is equivalent to the simultaneous avoidance of $(S+k)^{(S)}$ classical patterns of size $S+k$.
\end{theorem}
\begin{proof}
If a permutation $\pi$ avoids $q$, then $\pi$ avoids all classical patterns of the kind 
\[q' = p_{0,1} \cdots p_{0,r_{0}}q_{1}'p_{1,1}\cdots p_{1,r_{1}}q_{2}'\cdots q_{k-1}'p_{k-1,1} \cdots p_{k-1,r_{k-1}}q_{k}'p_{k,1} \cdots p_{k,r_{k}},
\]
where $p_{i,j}$ are $S$ distinct numbers from $1$ to $S+k$ and $q_{1}',q_{2}',\cdots ,q_{k}'$ are the remaining $k$ numbers from $1$ to $S+k$ and they are in the same relative order as the numbers $q_{1}q_{2}\cdots q_{k}$, i.e., $\textsf{fl}(q_{1}'$,$q_{2}'$,$\cdots$,$q_{k}') = q_{1}q_{2}\cdots q_{k}$. Indeed, any occurrence of such classical pattern would be an occurrence of $q$. The number of such classical patterns is $\frac{(S+k)!}{k!}$ since the relative order of exactly $k$ of the positions is fixed. Conversely, if $\pi$ avoids all the listed classical patterns, then it does not have an occurrence of $q$. Assume the opposite and take one such occurrence of $q$: $oc = \pi_{x_{1}}\pi_{x_{2}}\cdots \pi_{x_{k}}$. We know that $x_{1}>r_{0}$, $x_{2}-x_{1}>r_{1}$, etc. Select arbitrary $r_{0}$ letters of $\pi$, preceding $\pi_{x_{1}}$, arbitrary $r_{1}$ letters between $\pi_{x_{1}}$ and $\pi_{x_{2}}$ and so forth. You will obtain a subsequence $s = \pi_{y_{1}}\pi_{y_{2}}\cdots \pi_{y_{S+k}}$ of $\pi$ and $\textsf{fl}(s)$ must be one of the already listed classical patterns of size $S+k$ which is a contradiction.
\end{proof}
In fact, if we have some number of squares at the beginning or at the end of a distant pattern, we may consider the same distant pattern, but without those squares due to the following.
\begin{theorem}\label{th:sqBeginEnd}
For any $r_{1},r_{2}>0$ and a distant pattern $q$, we have
\begin{equation}
    |\operatorname{Av}_{n}(\square^{r_{1}}q\square^{r_{2}})| = n^{(r)}|\operatorname{Av}_{n-r}(q)|,
\end{equation}
where $r \coloneqq r_{1}+r_{2}$.
\end{theorem}
\begin{proof}
If $\sigma = \sigma_{1}\sigma_{2}\cdots \sigma_{n}\in \operatorname{Av}_{n}(\square^{r_{1}}q\square^{r_{2}})$, then
\begin{equation*}
\textsf{fl}(\sigma_{r_{1}+1}\cdots \sigma_{n-r_{2}})\in \operatorname{Av}_{n-r}(q).    
\end{equation*}
Conversely, any $\sigma = \sigma_{1}\sigma_{2}\cdots \sigma_{n}$ for which $\sigma_{1},\cdots,\sigma_{r_{1}}, \sigma_{n-r_{2}},\cdots, \sigma_{n}$ are $r$ arbitrary numbers in $[n]$ and for which $\textsf{fl}(\sigma_{r_{1}+1}\cdots \sigma_{n-r_{2}})\in \operatorname{Av}_{n-r}(q)$, would be such that $\sigma\in \operatorname{Av}_{n}(\square^{r_{1}}q\square^{r_{2}})$, since any possible occurrence of $q$ in $\sigma$ would have either less than $r_1$ other elements in front of it or less than $r_{2}$ elements after it.
\end{proof}
The theorem above tells us that it suffices to consider only classical DPs without $\square$ symbols at the beginning or at the end.

\section{Classical DPs of size 2}\label{sec:len2}
By theorem \ref{th:sqBeginEnd}, it suffices to consider $1\square 2$ and $2\square 1$ as the only DPs of size 2. One can obviously see that $|\operatorname{Av}_{n}(2 \square 1)| = |\operatorname{Av}_{n}(1 \square 2)|$, by applying the reverse map. Therefore we have only one Wilf-equivalence class here.
\begin{theorem}\label{th:fibo}
For $n\geq 3$, 
\begin{equation}
|\operatorname{Av}_{n}(2 \square 1)| = F_{n+1},    
\end{equation}
i.e., the $(n+1)$st Fibonacci number.
\end{theorem}
\begin{proof}
If $p_{1}p_{2}\cdots p_{n}\in \operatorname{Av}_{n}(2 \square 1)$, then either $p_{n} = n$ or $p_{n-1} = n$ since otherwise $n$ will participate in an inversion that is not of consecutive letters. If $p_{n} = n$ then $p_{1}p_{2}\cdots p_{n-1}$ must be in $\operatorname{Av}_{n-1}(2 \square 1)$, and for each permutation in $\operatorname{Av}_{n-1}(2 \square 1)$, we obtain a permutation in $\operatorname{Av}_{n}(2 \square 1)$ after appending the letter $n$ at the end. Thus, we have $|\operatorname{Av}_{n-1}(2 \square 1)|$ permutations in $\operatorname{Av}_{n}(2 \square 1)$ ending with $n$. If $p_{n-1} = n$, then we must have $p_{n} = n-1$ to prevent $n-1$ from forming a prohibited inversion with $p_{n}$. Thus, in this second case we must have $p_{n-1} = n$ and $p_{n} = n-1$ and the prefix $p_{1}p_{2}\cdots p_{n-2}$ must be a permutation in $\operatorname{Av}_{n-2}(2 \square 1)$. For each such permutation in $A_{n-2}(2 \square 1)$, we obtain a new one in $\operatorname{Av}_{n}(2 \square 1)$ by appending $n$ and then $n-1$ to it. Therefore $|\operatorname{Av}_{n}(2 \square 1)| = |\operatorname{Av}_{n-1}(2 \square 1)| + |A_{n-2}(2 \square 1)|$. It remains to note that $|A_{3}(2 \square 1)| = 3$ and $|A_{4}(2 \square 1)| = 5$.
\end{proof}
This basic result was proved in the seminal paper of Simion and Schmidt \cite{simion}, where they showed that $|\operatorname{Av}_{n}(123,132,213)| = F_{n+1}$ which as we explained (Theorem \ref{th:simultAv}) is equivalent to the fact above. 

One can consider more general settings for distant patterns and look at bigger values of the maximal distance between two consecutive letter of a pattern. Recall that $\operatorname{Av}_{n}(2 \square^{r} 1)$ is the set of all $p_{1}p_{2}\cdots p_{n}\in S_{n}$ with no inversion $(p_{i},p_{j})$, such that $|i-j|>r$. Apparently, $|\operatorname{Av}_{n}(2 \square^{m} 1)| = n!$ for $n \leq m
+1$ and $|\operatorname{Av}_{n}(2 \square^{0} 1)| = |\operatorname{Av}_{n}(21)| = 1$. The theorem below addresses the general case.
\begin{theorem}\label{th:bijection}
The permutations in $\operatorname{Av}_{n}(2\square^{r}1)$ are in one-to-one correspondence with the permutations in $S_n$ for which, when written in a cycle notation, any two elements in a cycle differ by at most $r$.
\end{theorem}
\begin{proof}
Let $X = \operatorname{Av}_{n}(2 \square^{r} 1)$ be the set of all permutations in $S_n$ that do not have inversions at distance greater than $r$ in their one-line notation representation. Let $Y$ be the set of those permutations in $S_n$ for which any two elements in the same cycle differ by at most $r$. We will describe a bijective map $f:X \to Y$. Consider $p = p_{1}p_{2}\cdots p_{n} \in X$. We will show how to obtain the standard form of $f(p)$ written in cycle notation, i.e., the minimal element of every cycle is at its first position and the cycles are ordered in increasing order of their minimal elements.
Below is the description of $f$: \newline
The number $p_1$ is at position 1, so let us look at the set of positions of all numbers with which $p_1$ is in inversion: $1,2,\cdots , p_{1} - 1$. Denote their positions with $i_{1},i_{2},\cdots , i_{p_{1}-1}$, respectively. These positions are not bigger than $r+1$, since $p\in X$. Then take $(1,i_{p_{1}-1},i_{p_{1}-2},\cdots , i_{1})$ to be the first cycle in the standard form of the cycle decomposition for $f(p)$. Then, let $j$ be the minimal number that is not already used in this cycle decomposition, and let the numbers $p_{j}-1, \cdots , p_{1}+1$ be at positions $j_{p_{j}-1},j_{p_{j}-2}, \cdots , j_{p_{1}+1}$. Take $(j,j_{p_{j}-1},j_{p_{j}-2}, \cdots , j_{p_{1}+1})$ as the next cycle in the standard form of the cycle decomposition for $f(p)$ and continue in the same way afterwards. Note that the length of some of those cycles might be $1$.

Here are two examples: 
\begin{itemize}
    \item If $n=9,r=3$ and $p = 352149867$, then $f(352149867) = (134)(25)(6798)$. 
    \item If $n=8,r=4$ and $p = 41352867$, then $f(41352867)=(1352)(4)(687)$. 
\end{itemize}

Obviously, $f$ maps each $\sigma \in X$ to a permutation $f(\sigma)$ such that any two numbers in the same cycle of $f(\sigma)$ differ by at most $r$, since these two numbers correspond to indices of two numbers, in the one-line notation of $\sigma$, which are in inversion in $\sigma\in X$. To prove that $f$ is indeed a bijection, we will describe its inverse. Consider $\pi\in Y$ in its standard cycle decomposition form. If the first cycle of $\pi$ is $(\pi_{1}\pi_{2}\cdots \pi_{i_{1}})$, then put the number $i_{1}$ in the first place, then $i_{1}-1$ at position $\pi_{2}$, $i_{1}-2$ at position $\pi_{3}$ and so on. The number $1$ will be placed at position $\pi_{i_{1}}$. Note also that $\pi_{1}$ is always $1$. Next, go to the next cycle $(\pi_{j_{1}}\pi_{j_{2}}\cdots \pi_{j_{i_{2}}})$. We will determine the positions of the next $i_{2}$ numbers: $i_{1}+1,i_{1}+2,\cdots ,i_{1}+i_{2}$. We can see that $\pi_{j_{1}}$ must be the least integer not occurring in the first given cycle. We will place at this position, the number $i_{1}+i_{2}$. Then, $i_{1}+i_{2}-1$ should be placed at position $\pi_{j_{2}}$, $i_{1}+i_{2}-2$ at position $\pi_{j_{3}}$ and so on. One can use the two given examples above for verification.
\end{proof}

The sequences of the numbers $|\operatorname{Av}_{n}(2 \square^{r} 1)|$, for different fixed values of $n$ and $r$ are respectively rows and columns of the table described in \cite[A276837]{OEIS}. For example, in the case $r=2$, if we denote $a_{n} = |\operatorname{Av}_{n}(2 \square^{2} 1)|$, then it is not hard to prove the recurrence 
\begin{equation}
a_{n} = a_{n-1} + a_{n-2} + 3.a_{n-3} + a_{n-4}.
\end{equation}
\subsection{Recurrence formula for $F_{n+1} = |\operatorname{Av}_{n}(1\square 2)|$}
\label{sec:rec}
In an attempt to obtain a general enumeration approach when dealing with DPs, we tried to use a technique that is described in the current subsection. The technique helped us to obtain a recurrence formula for the number of permutations avoiding the distant pattern $1 \square 2$, i.e., a new recurrence formula for the Fibonacci numbers (see Theorem \ref{th:fibo}). The same technique can be used when trying to enumerate the set of avoiders for other DPs.

The idea is that almost all permutations containing a given distant pattern can be obtained by first taking a permutation containing the corresponding classical pattern and then inserting additional numbers between some of the letters (where we have the $\square$ symbol) for a certain occurrence of this classical pattern. Let us describe this more concretely with the following algorithm that we will use for the pattern $1 \square 2$.
\begin{description}
\item [\textbf{Algorithm 1}]\
\begin{enumerate}
\item[1.] For a given $n\geq 3$ and $j\in [n]$, take any $\pi\in S_{n-1}\setminus \operatorname{Av}_{n-1}(12)$.

\item[2.] Find the leftmost $1$ that is part of a classical $12$-pattern and insert the number $j$ immediately after it.

\item[3.] Increase by $1$ the numbers $j,j+1,\ldots ,n-1$, except the $j$ that we just inserted (unless $j=n$, $\pi$ contains another $j$).
\end{enumerate}
\end{description}

This algorithm defines a map $g: A_{n-1} \to B_{n}$, where
$A_{n-1} = (S_{n-1}\setminus \operatorname{Av}_{n-1}(12))\times [n]$ and $B_{n} = S_{n} \setminus \operatorname{Av}_{n}(1 \square 2)$. Here is an example: $g(3412,2) = 42513$. The leftmost occurrence of the pattern $12$ in $3412$ is by the first two letters, $3$ and $4$. Therefore we insert $j=2$ immediately after the letter $3$ and then increase the $2,3$ and $4$ in the original permutation. Note that the added number $j$ always keep its value in the final image.
We will first need to show that this map is not far from being injective.
\begin{theorem} \label{th:almostInjective}
No permutation in $B_{n} = S_{n} \setminus \operatorname{Av}_{n}(1 \square 2)$, the range of the map $g$, is the image of more than two different elements of $A_{n-1} = (S_{n-1}\setminus \operatorname{Av}_{n-1}(12))\times [n]$.
\end{theorem}
\begin{proof}
Assume the opposite. Let $\pi  = g(\pi_{1},j_{1}) = g(\pi_{2},j_{2}) = g(\pi_{3},j_{3})$ for three different tuples $(\pi_{1},j_{1}), (\pi_{2},j_{2}), (\pi_{3},j_{3}) \in A_{n} = (S_{n-1}\setminus \operatorname{Av}_{n}(12))\times [n]$. We can see that $j_{1},j_{2}$ and $j_{3}$ must be different since if two of them, say $j_{1}$ and $j_{2}$, are equal then obviously $\pi_{1} = \pi_{2}$ and we will not have different tuples. Now, we know that without loss of generality $1 < c_{j_1} < c_{j_2} < c_{j_3}$ are three different positions for the three different numbers $j_{1},j_{2},j_{3}$ in the final permutation $\pi$. By step 2 of Algorithm 1, after removing $j_1$ from $\pi$, the first occurrence of the classical pattern 12, should be some $\pi_{x}\pi_{y}$, where $x = c_{j_1}-1$. Similarly, after removing $j_2$, the first such occurrence should begin at position $c_{j_{2}}-1 > c_{j_{1}}-1$, but this is only possible if the position $y = c_{j_2}$ since if this is not the case then $\pi_{x}\pi_{y}$ would be an occurrence of 12 that begins before position $c_{j_{2}}-1$. However, after removing $j_3$ from $\pi$ (note that $c_{j_{3}}>c_{j_{2}}$), the first occurrence of 12 should begin at position $c_{j_{3}}-1 > c_{j_1}-1$. Contradiction.
\end{proof}
There are many permutations in $B_{n}$ which are the image of $g$ for two different elements of $A_{n-1}$. An example is $3142\in B_{4}$ since $g(312,4) = g(231,1) = 3142$. The next fact that we will need gives the number of these permutations.
\begin{theorem}\label{th:exactlyTwo}
The number of permutations $\omega$ in $B_{n}$ which are an image of exactly two different elements of $A_{n-1}$, after applying the map g, is given by the sum
\begin{equation}
    \sum\limits_{j=3}^{n-1} (j-2)(n-j)(n-j)!.
\end{equation}
\end{theorem}
\begin{proof}
Let $\pi = g(\pi_{1},x) = g(\pi_{2},y)$ for $\pi_{1}, \pi_{2} \in S_{n-1}\setminus \operatorname{Av}_{n-1}(12)$ and $x,y \in [n]$, where the tuples $(\pi_{1},x)$ and $(\pi_{2},y)$ are different. We saw in the proof of Theorem \ref{th:almostInjective} that $x$ and $y$ must be different. Let us denote the positions of $x$ and $y$ in $\pi$ with $i$ and $j$ respectively. Without loss of generality, let $i<j$. We know that after removing $y = \pi_{j}$ from $\pi$, then $\pi_{j-1}\pi_{j+k}$, for some $k \geq 1$, is the first occurrence of the classical pattern 12. Therefore, we should have $\pi_{1}>\pi_{2}> \cdots > \pi_{j-1}$. Since, if we remove $x = \pi_{i}$ from $\pi$, then $\pi_{i-1}\pi_{j}$ must be the first occurrence of $12$, it follows that we must have $\pi_{1}>\pi_{2}> \cdots > \pi_{i-2} > \pi_{j} > \pi_{i-1} > \cdots >\pi_{j-1}$. In other words, the number $\pi_{j}$ is between $\pi_{i-2}$ and $\pi_{i-1}$. Otherwise, we would have a $12-$occurrence ending at $\pi_{j}$ that starts before position $i-1$. We also have that $\pi_{j-t}>\pi_{j+l}$, for any $t = 2,\cdots , j-1$ and $l = 1,\cdots , n-j$, because otherwise when removing $\pi_{j}$ from $\pi$, a $12$-occurrence starting before $\pi_{j-1}$ will be present. 

In order to determine $\pi$ completely, we must have the relations between the $n-j+1$ numbers $\pi_{j-1},\pi_{j+1},\pi_{j+2},\cdots , \pi_{n}$. The only constraint that we have is that $\pi_{j-1}$ is not the biggest among them. Thus, when $i$ and $j$ are fixed, we always have $(n-j+1)! - (n-j)!$ possible ways to write $\pi$. Therefore, the number of different permutations $\pi \in S_{n}$ that are an image for two different tuples is 
\begin{center}
$\sum\limits_{j=3}^{n-1}\sum\limits_{i=2}^{j-1}[(n-j+1)! - (n-j)!] = \sum\limits_{j=3}^{n-1} (j-2)(n-j)(n-j)!$.
\end{center}
Each term in the latter sum gives us the number of permutations in $\pi\in B_{n}$, where $\pi = g(\pi_{1},x) = g(\pi_{2},y)$ for some $\pi_{1}, \pi_{2} \in S_{n-1}\setminus \operatorname{Av}_{n-1}(12)$ and $x,y \in [n]$, where $x<y$ and $y$ is at position $j$ in $\pi$.
\end{proof}
As we have seen that no permutation in $B_{n}$ is counted more than two times, it remains to obtain the number of permutations in $B_{n}$ that are not an image of $g$ for any permutation in the set of tuples $A_{n-1}$.
\begin{theorem}\label{th:missing}
The number of permutations in the set $B_{n} \setminus g(A_{n-1})$ is:
\begin{equation}
    \sum\limits_{k=3}^{n-2} (F_{n-k+1}-1)k(k-2)(k-2)! ,
\end{equation}
where $F_{i}$ denotes the $i$-th Fibonacci number.
\end{theorem}
\begin{proof}
An example of a permutation in $B_{n}$ that cannot be obtained as an output of the function $g$ (i.e., with Algorithm 1) for any input in $A_{n-1}$ is the permutation $45132$. The reason is that before the first occurrence of the distant pattern  $1 \square 2$, there exist an occurrence of the classical pattern $12$ (which is not an occurrence of $1 \square 2$). We want to obtain a formula for all permutations in $B_{n}$ having this property. Consider one such permutation $\pi = \pi_{1}\pi_{2}\cdots \pi_{n}$ and let the first occurrence of $1 \square 2$ be $\pi_{j}\pi_{j+v}$ for some $j\geq 1$, $v\geq 2$, and $j+v \leq n$. Since, this is the first such occurrence, observe that $\pi_{i}>\pi_{j+d}$, for any $i<j$ and $d\geq 0$. Otherwise we would have another occurrence of $1 \square 2$, preceding $\pi_{j}\pi_{j+v}$. Thus $\pi' = \pi_{1}\cdots \pi_{j-1}$ must be a permutation of the $j-1$ numbers $n-j+2, \ldots , n$ and $\pi'' = \pi_{j}\ldots \pi_{n}$ is simply a permutation of $1, \cdots , (n-j+1)$ for which $\pi_{j} < \pi_{j+v}$ for some $v = 1 \cdots n-j$. This means that $\pi'$ avoids $1 \square 2$, but contains the classical pattern 12. Therefore the number of possibilities for $\pi'$ is $|\operatorname{Av}_{j-1}(1 \square 2)| - |\operatorname{Av}_{j-1}(12)| = F_{j}-1$ since $|\operatorname{Av}_{j-1}(1 \square 2)| = F_{j}$ (Theorem \ref{th:fibo}), $|\operatorname{Av}_{j-1}(12)| = 1$ and $A_{j-1}(12) \subset A_{j-1}(1 \square 2)$. Now, let us denote $k = n-j+1$, for clarity. For $\pi''$, we can see that it could be any $k$-permutation except that it could not start with $k$ or with $(k-1)k$ since if this is the case, then $\pi''$ will not start with an occurrence of the $1 \square 2$ pattern. The latter means that the possible values for $\pi''$ are exactly $k!-(k-1)!-(k-2)! = k(k-2)(k-2)!$. Summing over $k$, we obtain the given formula.
\end{proof}

Now, we are ready to derive the recurrence formula that we want. Note that $|A_{n-1}| = |(S_{n-1}\setminus \operatorname{Av}(12))\times [n]| = ((n-1)!-1)\cdot n = n!-n$ gives the number of permutations in $B_{n}$ that are the image of the map $g$ for exactly one tuple in $A_{n-1}$. Theorems \ref{th:exactlyTwo} and \ref{th:missing} give the number of permutations being the image of $g$ for $2$ and $0$ tuples in $A_{n-1}$, respectively. We also know that $|B_{n}| =  |S_{n} \setminus \operatorname{Av}_{n}(1 \square 2)| = n! - F_{n+1}$. Thus using inclusion-exclusion we have:
\begin{equation*}
    (n!-F_{n+1}) - \sum\limits_{k=3}^{n-2} (F_{n-k+1}-1)k(k-2)((k-2)!) = (n!-n) - \sum\limits_{j=3}^{n-1} (j-2)(n-j)(n-j)!.
\end{equation*}

After simplifying, we obtain the following recurrence formula for the Fibonacci numbers and respectively for the number of permutations avoiding the distant pattern $1 \square 2$ (or $2 \square 1$):
\begin{equation}\label{eq:fiboRec}
    |\operatorname{Av}_{n}(1\square 2)| = F_{n+1} = n + \sum_{k=1}^{n-3} (n-(k+2)F_{n-(k+1)})\cdot k\cdot k!.
\end{equation}

\section{Classical DPs of size 3} \label{sec:len3}
As we can infer from Theorem \ref{th:simultAv}, finding a closed formula for the avoidance set of a distant pattern becomes more complicated as its size increases, because the number of classical patterns that must be simultaneously avoided increases, as well. In this section, we describe some already established results on the DPs of size $3$ with one square ($xy\square z$) and two squares ($x\square y\square z$). Then, we discuss an approach that we have used to obtain the generating function for $|\operatorname{Av}_{n}(1\square 3\square 2)|$, which represents one of the two different Wilf-equivalent classes for patterns of the latter kind.  

\subsection{Patterns of the kind $xy\square z$}
\label{sec:polygons}
Consider the patterns $xy\square z$ and $x\square yz$, for some permutation $xyz\in S_{3}$. The thesis of Firro \cite{thesis} and two related works \cite{MansourPolygons, MansourFuncEq} give the formula
\begin{equation}\label{eq:firro123}
    |\operatorname{Av}_{n}(12 \square 3)| = \sum\limits_{k \geq 0} \frac{1}{n-k}\binom{2n-2k}{n-1-2k}\binom{n-k}{k}.
\end{equation}
The same thesis gives two bijections between $12 \square 3$-avoiding permutations and odd-dissections of a given $(n+2)$-gon, which are dissections with non-crossing diagonals so that no $2m$-gons ($m > 1$) appear \cite[A049124]{OEIS}. In fact, it turns out that this is cardinality of the avoidance set for any pattern of the kind $xy \square z$ or $x \square yz$ \cite{thesis}. We know that all classical patterns in $S_{3}$ are avoided by the same number of permutations, namely the Catalan numbers. One might suspect that whenever two classical patterns $p,q\in S_{k}$ are Wilf-equivalent, then inserting a square at the same place in $p$ and $q$ will produce two Wilf-equivalent distant patterns. The computer simulations shows that the former seems to be true for the Wilf-equivalent patterns $\{1234,1243,2143\}$. We have formulated this conjecture in Section \ref{sec:open}.

It was shown in \cite{thesis} that if $xyz\in S_{3}$, then inserting a square between $x$ and $y$ or between $y$ and $z$ always gives us two Wilf-equivalent patterns. It is worth noting that we do not have a similar fact when considering patterns of bigger size. For example, $|\operatorname{Av}_{7}(1 \square 234)| = 3612 \neq 3614 = |\operatorname{Av}_{7}(12 \square 34)|$.

\subsection{Patterns of the kind $x\square y\square z$}
The inverse and the complement map give us at most two Wilf-equivalent permutation classes : $\{\operatorname{dist}_{1}(p)\mid p = 132,231,213,312\}$ and $\{\operatorname{dist}_{1}(p)\mid p = 123,312\}$. Unlike the case of classical patterns in which these are, in fact, one class \cite{simion}, here, these classes are different.
 \begin{theorem} (\cite{hopkins})\label{th:123IsLargest}
 For $n > 5$, 
 \begin{equation}
 |\operatorname{Av}_{n}(\operatorname{dist}_{1}(123))| > |\operatorname{Av}_{n}(\operatorname{dist}_{1}(132))|.
 \end{equation}
 \end{theorem}
 The theorem above is a special case of a result of Hopkins and Weiler \cite[Theorem 3]{hopkins}. In that work they extend the result of Simion and Schmidt that $|\operatorname{Av}_{n}(123)| = |\operatorname{Av}_{n}(132)|$  from permutations on a totally ordered set to a similar result for pattern avoidance in permutations on partially ordered sets. In particular, they show that $|\operatorname{Av}_{P,n}(132)| \leq |\operatorname{Av}_{P,n}(123)|$ for any poset $P$, where $\operatorname{Av}_{P,n}(q)$ is the number of $n$-permutations on the poset $P$ avoiding the pattern $q$. Furthermore, they classify the posets for which equality holds. Here, we state the corollary of their result generalizing Theorem \ref{th:123IsLargest}, as formulated by the authors.
 \begin{theorem} (\cite{hopkins})
 For $r \geq 0$ and $n \geq 1$, we have
 \begin{equation}
|\operatorname{Av}_{n}(\operatorname{dist}_{r}(123))| \geq |\operatorname{Av}_{n}(\operatorname{dist}_{r}(132))|,
\end{equation} 
 with strict inequality if and only if $r \geq 1$ and $n \geq 2r+4$.
 \end{theorem}
 Note that in the case $n=2r+3$, $|\operatorname{Av}_{2r+3}(\operatorname{dist}_{r}(123))| = |\operatorname{Av}_{2r+3}(\operatorname{dist}_{r}(132))|$ since there is only one triple of positions where each of these two patterns can occur in a $(2r+3)$-permutation, namely the positions $1,r+2$ and $2r+3$. So for each such occurrence, we can exchange the elements at positions $r+2$ and $2r+3$ to get an occurrence of the other pattern. A similar statement about consecutive patterns was first proved in \cite{Elizalde2003} with a simple injection. It states that $|\operatorname{Av}_{n}(\underline{123})| > |\operatorname{Av}_{n}(\underline{132})|$ for every $n \geq 4$.
 The listed facts imply that the monotonic pattern $123$ is avoided more frequently than $132$ when we have two gaps of size exactly $0$ between the letters in each occurrence of the two patterns, or when the minimal constraint for each  gap is some fixed positive number. However, when patterns with all possible gap sizes must be avoided, we have an equality since $|\operatorname{Av}_{n}(123)| = |\operatorname{Av}_{n}(132)|$. We address this surprising fact in the next section.
 
 Along those lines is another work of Elizalde \cite{Elizalde2012} on consecutive patterns, where he generalizes \cite{Elizalde2003} by proving that the number of permutations avoiding the monotone consecutive pattern $\underline{12 \cdots m}$ is asymptotically larger than the number of permutations avoiding any other consecutive pattern of size $m$. He also proved there that $|\operatorname{Av}_{n}(\underline{12 \dots  (m-2)m(m-1)})|$ is asymptotically smaller than the number of permutations avoiding any other consecutive pattern of the same size. Similar conjectures can be formulated for distant patterns (see Section \ref{sec:open}). 
 
 \subsubsection{The pattern $1 \square 3 \square 2$}
In this subsection, we will roughly describe an approach that one can use to find the generating function $G(x) = \sum\limits_{n\geq 0}^{}|\operatorname{Av}_{n}(q)|x^{n}$, where $q=1\square 3\square 2 = \operatorname{dist}_{1}(132)$. All the details about this proof and the technique, that we use, will be described in a separate, forthcoming, article. Here, we will sketch that proof. To do that, we will need to define the following sets of permutations:
 \begin{center}
     $\mathbb{H}_1 \coloneqq \{\pi \mid \pi \in \operatorname{Av}(q)$, $|\pi|\geq 1$ and $\pi$ does not have an occurrence of $1\square\underline{32}$ ending at the last position of $\pi$$\}$, and
 \end{center}
 \begin{center}
     $\mathbb{H}_2 \coloneqq \{\pi \mid \pi \in \operatorname{Av}(q)$, $|\pi|\geq 1$ and $\pi$ does not have an occurrence of $\underline{13}\square2$ beginning at the first position of $\pi$ $\}$.
 \end{center}
 Let us also denote the corresponding generating functions with 
 \begin{center}
     $H_{i}(x) = \sum\limits_{k=1}^{\infty}h_{i}(k)x^{k}$,
 \end{center} 
 where $h_{i}(n)$ is the number of permutations of size $n$ in $\mathbb{H}_i$ ($i=1,2$). Now, we can describe a useful decomposition for the permutations in $\operatorname{Av}_{n}(q)$ which is similar, but more complicated, to the one given in \cite{thesis} for the permutations in $\operatorname{Av}_{n}(13\square 2)$. 
 
 \begin{theorem} \label{th:132decompose}
 For all $n \geq 1$, $\pi = \alpha n \beta \in \operatorname{Av}_{n}(q)$ if and only if:
 \begin{enumerate}[label=(\roman*)]
     
     \item $\alpha > \beta$, $\alpha,\beta \in \operatorname{Av}(q)$
     
     \item $\alpha \ngtr \beta$, but $\alpha' > \beta'$, where $\alpha = \alpha't_{1}$ and $\beta = t_{2}\beta'$ for some $t_{1},t_{2}\in [n-1]$. and one of the following holds:
     \begin{enumerate}[label=\arabic*.]
         \item $t_{1}>\beta'$, $t_{2}<\alpha'$, $t_{1}<t_{2}$ and $\alpha',\beta' \in \operatorname{Av}(q)$
         
         \item $t_{1}>\beta'$, $t_{2}\nless\alpha'$, $\beta'\in \operatorname{Av}(q)$ and $\sigma = \alpha't_{1}t_{2}\in\mathbb{H}_1$ with $t_{2}$ not being the smallest element in $\sigma$ and not being the second smallest, after $t_{1}$.
         
         \item $t_{1}\ngtr\beta'$, $t_{2}<\alpha'$, $\alpha'\in \operatorname{Av}(q)$ and $\sigma = t_{1}t_{2}\beta'\in\mathbb{H}_2$ with $t_{1}$ not being the biggest element in $\sigma$ and not being the second biggest, after $t_{2}$.
         
         \item $t_{1}\ngtr\beta'$, $t_{2}\nless\alpha'$, $\sigma_{1} = \alpha't_{2}\in \mathbb{H}_1$ with $t_{2}$ not being the smallest element in $\sigma_{1}$ and $\beta' = x\beta''$, where $x>t_{1}>\beta''$ and $\beta''\in \operatorname{Av}(q)$.
         
     \end{enumerate}
 \end{enumerate}
 \end{theorem}
In this paper, we omit the proof of the fact above. The described decomposition gives us the next result almost directly
\begin{theorem}\label{th:G}
\begin{equation}\label{eq:G}
\begin{split}
G(x) = 1 + G(x)(xH_{1}(x) + xH_{2}(x) + x^{3}H_{1}(x)) + G^{2}(x)(x-2x^{2}-x^{3}-x^{4}).
\end{split} 
\end{equation}
 \end{theorem}
In order to obtain $G(x)$, we first express $H_{1}(x)$ as a function of $H_{2}(x)$  and $G(x)$. Then we express $H_{2}(x)$ as a function of $G(x)$ using the block-decomposition method \cite{MansourVain} and an additional fact. These are the main ingredients of our approach. Extensive case analysis and inclusion-exclusion arguments are additionally used. As a result, we obtain a system of two equations each of which is a polynomial of $x$, $G(x)$ and $H_{2}(x)$. We eliminate $H_{2}(x)$ to obtain an equation $P(x,G(x)) = 0$, where $P$ is a polynomial of $G(x)$ with coefficients that are polynomials of $x$. $P$ has $170$ terms with the term of highest total degree being $x^{27}G^{12}$. One could use a generalization of the Lagrange inversion formula discussed in the work of Baderier and Drmota \cite{banderier} to get a closed form expression for the coefficients of $G(x)$ which is our final goal.

\section{Vincular distant patterns}
In this section, we consider a particular kind of vincular distant patterns of size 3. The goal will be to compare the number of permutations avoiding the different kinds of 123 and 132 patterns. 
\subsection{Patterns of the form $\underline{ab}\square c$ and $a\square\underline{bc}$}\label{sec:vinc}
  There are $12$ patterns of this kind, and the reverse and complement maps give at most three Wilf-equivalence classes listed below. As we will show, these turn out to be different.

\begin{table}[h!]
\centering
\captionsetup{position=below}
\begin{tabular}{||c|c|c||} 
 \hline
 Class 1 & Class 2 & Class 3 \\ [0.5ex] 
 \hline\hline
$\underline{12} \square 3$ &
$1 \square \underline{32}$ &
$\underline{13} \square 2$ \\

$\underline{32} \square 1$ &
$\underline{21} \square 3$ &
$\underline{31} \square 2$ \\

$1 \square \underline{23}$ &
$\underline{23} \square 1$ &
$2 \square \underline{31}$ \\ 

$3 \square \underline{21}$ &
$3 \square \underline{12}$ &
$2 \square \underline{13}$ \\ [1ex] 
 \hline
\end{tabular}
\caption{The Wilf-equivalent classes of $\underline{ab}\square c$ and $a\square\underline{bc}$ patterns.}
\label{table:1}
\end{table}

We will first find a recurrence for the pattern $\underline{12} \square 3$:

\begin{theorem}\label{th:vinc123}
If $a_{n} = |\operatorname{Av}_{n}(\underline{12} \square 3)|$, then $a_{n} = n!$ for $0 \leq n \leq 3$, and for $n \geq 4$ we have 

\begin{align*}
    a_{n} = a_{n-1} + (n-1)a_{n-2} + \frac{(n+1)(n-2)}{2}a_{n-3} + \\ \sum\limits_{i=4}^{n-1}\left(\binom{n}{i-1}-1\right)a_{n-i} + (n-1)
\end{align*}    

\end{theorem}
\begin{proof}
Let $q = \underline{12} \square 3$ and let $\pi = \pi_{1}\pi_{2} \cdots \pi_{n}$ be a permutation of $[n]$ that avoids $q$. We will consider five cases for the position of the number $n$ in $\pi$. Denote this position with $i$, so $\pi_{i} = n$.
\begin{enumerate}[label={Case \arabic*.}]
    \item \hspace{3mm}$i=1$: $\pi  = n\pi_{2}\cdots\pi_{n}$\\
    In this case, $n$ will not participate in any occurrence of $q$ since it can only be the first letter in such occurrence. Thus since $\pi$ avoids $q$ then $\pi_{2} \cdots \pi_{n}$ must avoid $q$. There are $a_{n-1}$ such permutations $\pi_{2}\cdots\pi_{n}\in S_{n-1}$. 
    \item \hspace{3mm}$i=2$: $\pi  = \pi_{1}n \cdots \pi_{n}$\\
    Here, $n$ cannot participate in any occurrence of $q$. Neither can $\pi_{1}$, because it could participate only together with $\pi_{2} = n$. Then $\pi_{1}$ can be any of the remaining $n-1$ numbers. Regardless of the choice of $\pi_{1}$, one would have $a_{n-2}$ ways to choose the order of the remaining $n-2$ letters since $\textsf{fl}(\pi_{3}\cdots \pi_{n})$ must avoid $q$. This gives $(n-1)a_{n-2}$ ways to obtain $\pi$.  
    \item \hspace{3mm}$i=3$: $\pi  = \pi_{1}\pi_{2}n\cdots \pi_{n}$\\
    The number $n$ cannot be part of a $q$-occurrence, again. Therefore if $n-1$ is in an occurrence of $q$, then it must be the last letter (the '3'). Let $j = \pi^{-1}(n-1)$ be the position of $n-1$ in $\pi$.
    \begin{enumerate}[label={Case 3\alph*.}]
        \item $j = 1$ or $j = 2$ \\
        None of the first three elements of $\pi$ will be part of any occurrence of $q$. Thus we have $2(n-2)a_{n-3}$ permutations $\pi\in \operatorname{Av}_{n}(q)$ with $i=3$ and $j = 1$ or $j = 2$, since we can choose the position, $j$, of $n-1$ in 2 ways and the other of the first 2 letters in $n-2$ ways. The rest of the permutation must avoid $q$ and there are $a_{n-3}$ possibilities for that. We get a $q$-avoiding permutation in all of these cases.
        \item $j>3$ \\
        Here, since $\pi$ avoids $q$, we must have $\pi_{1}>\pi_{2}$, because otherwise $\pi_{1}\pi_{2}\pi_{j}$ would be a $q$-occurrence. We can determine $\pi_{1}$ and $\pi_{2}$ in $\binom{n-2}{2}$ ways. The number of ways to determine $\pi_{4} \cdots \pi_{n}$ would be again $a_{n-3}$, despite knowing that $n-1$ will be one of these letters, simply because this part of $\pi$ must avoid $q$ and because once we have $\pi_{1}$, $\pi_{2}$ and $\pi_{3}$ fixed, this part will correspond to a permutation in $\operatorname{Av}_{n-3}(q)$. 
    \end{enumerate}
    \item \hspace{3mm}$3<i<n$. $\pi  = \pi_{1}\pi_{2}\cdots n\cdots \pi_{n}$\\
    Since $\pi$ avoids $q$, the numbers $\pi_{1}$, $\pi_{2}, \cdots \pi_{i-2}$ must be in decreasing order. We have three subcases for the position $j = \pi^{-1}(n-1)$.
    \begin{enumerate}[label={Case 4\alph*.}]
        \item $j=i-1$: $\pi  = \pi_{1}\pi_{2}\cdots \pi_{i-2}(n-1)n\cdots \pi_{n}$\\
        The numbers $\pi_{1}, \cdots ,\pi_{i-2}$ must be in decreasing order since $\pi\in \operatorname{Av}(q)$. Once we have chosen these $i-2$ numbers of $\pi$ then neither $\pi_{i-1} = n-1$ nor $\pi_{i} = n$ could participate in a $q$-occurrence and any ordering of the last $n-i$ numbers that avoids $q$ would give us a different $q$-avoider $\pi$. This gives $\binom{n-2}{i-2}a_{n-i}$ permutations for this case.
        \item $j<i-1$ (in fact, $j=1$): $\pi  = (n-1)\pi_{2}\cdots n\pi_{i+1}\cdots\pi_{n}$\\
        This would imply that $j=1$ since $\pi_{1}, \cdots , \pi_{i-2}$ are in decreasing order. If $\pi_{i-2}>\pi_{i-1}$, then we can select $\pi_{2}, \cdots ,\pi_{i-1}$ in $\binom{n-2}{i-2}$ ways which gives $\binom{n-2}{i-2}a_{n-i}$ more $q$-avoiding permutations. Slightly more attention is required for the subcase $\pi_{i-2}<\pi_{i-1}$. In order to avoid $q$, all of $\pi_{i+1}, \cdots , \pi_{n}$ must be smaller than $\pi_{i-1}$, because otherwise $\pi_{i-2}\pi_{i-1}\pi_{k}$ would be a $q$-occurrence for some $k > i$. Now, we should calculate how many different permutations $\pi$ satisfy the described conditions. For clarity, one may look at Figure \ref{fig:vinc123} that visualizes the order of the elements in one such $\pi$.
        \begin{figure}
            \centering
        \begin{tikzpicture}[scale=1.25]
    \node[below] at (-3,2) {$n-1=\pi_{1}$}; 
    \filldraw [black] (-2,2) circle (2.5pt);
    \node[below] at (-2,1.5) {$\pi_{2}$}; 
    \filldraw [black] (-1.5,1.5) circle (2.5pt);
    \filldraw [black] (-1.2,1.2) circle (1pt);
    \filldraw [black] (-0.9,0.9) circle (1pt);
    \node[below] at (-1.2,0.6){$\pi_{i-2}$};
     \filldraw [black] (-0.6,0.6) circle (2.5pt);
     \filldraw [black] (-0.3,1.75) circle (2.5pt);
     \node[above] at (0,2.25){$n = \pi_{i}$};
     \filldraw [black] (0,2.25) circle (2.5pt);
     
     \filldraw [black] (0.3,1.35) circle (2.5pt);
     \filldraw [black] (0.45,0.75) circle (2.5pt);
     \filldraw [black] (0.7,0.95) circle (1pt);
     \filldraw [black] (0.9,0.95) circle (1pt);
     \filldraw [black] (1.1,0.95) circle (1pt);
     \filldraw [black] (1.4,0.8) circle (2.5pt);
     \filldraw [black] (1.55,0.5) circle (2.5pt);
     \filldraw [black] (1.7,1.3) circle (2.5pt);
    \end{tikzpicture}
            \caption{sketch of the order of the elements of $\pi$ in Case $4b.$}
             \label{fig:vinc123}
        \end{figure}
        
        We claim that the number of these permutations is $(\binom{n-2}{i-3}-1)a_{n-i}$. Indeed, we can first choose the last $n-i$ numbers $\pi_{i+1},\pi_{i+2}, \cdots ,\pi_{n}$, and the number $\pi_{i-1}$. Those are the unlabeled elements on Figure \ref{fig:vinc123}. We can do that in $\binom{n-2}{n-i+1} = \binom{n-2}{i-3}$ ways. Out of these choices, only the one where we have selected the smallest numbers, $1,2, \cdots , n-i+1$, would force $\pi_{i-2}>\pi_{i-1}$ which we do not want to happen, so we exclude this single choice. For all the other choices, we simply have that the biggest number among the chosen has to be $\pi_{i-1}$ and the other $n-i$ chosen numbers can be ordered in $a_{n-i}$ ways at positions $i+1$, $i+2$, $\cdots$, $n$. The unchosen $i-3$ numbers are ordered decreasingly after $\pi_{1} = n-1$, at positions $2,3, \cdots , i-2$.
        \item $j > i$ \\
        In this case, the numbers $\pi_{1}, \cdots , \pi_{i-1}$ must all be in decreasing order. Thus, it suffices just to choose which are they and choose the numbers in the remaining part of the permutation, i.e., we have $\binom{n-2}{i-1}a_{n-i}$ permutations here.
    \end{enumerate}
    \item \hspace{3mm}$i=n$. $\pi  = \pi_{1}\pi_{2}\pi_{3}\cdots n$\\
    Again, the numbers $\pi_{1}, \cdots , \pi_{n-2}$ must be in decreasing order, so it suffices to choose $\pi_{n-1}$ in $n-1$ ways.
\end{enumerate}
It remains to observe that in Case 4, after summing the number of $q$-avoiding permutations for the three subcases, we get 
\begin{equation*} \label{eq1}
\begin{split}
\left(\binom{n-2}{i-2} + \binom{n-2}{i-2} + \left(\binom{n-2}{i-3}-1\right) + \binom{n-2}{i-1}\right)a_{n-i} & = \\ \left(\binom{n-2}{i-2} + \left(\binom{n-2}{i-3}-1\right) + \binom{n-1}{i-1}\right)a_{n-i} & = \\ \left(\binom{n-1}{i-2} + \binom{n-1}{i-1} -1\right)a_{n-i} = \left(\binom{n}{i-1} -1\right)a_{n-i}
\end{split}
\end{equation*}
\end{proof}
The first few elements of the sequence $|\operatorname{Av}_{n}(\underline{12} \square 3)|$ for $n\geq 4$ are 
\begin{center}
$20,75,316,1464,7359,39815,230306.$
\end{center}
This is not part of any sequence in OEIS. 

This enumerates avoidance for Class 1 patterns in Table 1. Similar recurrence can be found for the patterns in Class 2. We give it below using the pattern $1 \square \underline{32}$.
\begin{theorem}\label{th:vinc132}
If $b_{n} = |\operatorname{Av}_{n}(1 \square \underline{32})|$, then $b_{n} = n!$ for $0 \leq n \leq 3$ and for $n \geq 4$ we have 
\begin{equation}\label{eq:vinc132}
\begin{split}
    b_{n} = b_{n-1} + (n-1)b_{n-2} + \frac{(n+1)(n-2)}{2}b_{n-3} + \\ \sum\limits_{i=2}^{n-3}\left(i\binom{n-2}{i}+\binom{n-1}{i-1}\right)b_{i-1} + (n-1).
\end{split}
\end{equation}
\end{theorem}
The first few elements of the sequence $|\operatorname{Av}_{n}(1 \square \underline{32})|$ for $n\geq 4$ are
\begin{center}
$20,76,326,1544,7954,44164,262456$.    
\end{center}
 This is not part of any sequence in OEIS.

Theorems \ref{th:vinc123} and \ref{th:vinc132} differ only in the sums in their right-hand sides. Applying the complement map after the reverse map, we see that $|\operatorname{Av}_{n}(\underline{12} \square 3)| = |\operatorname{Av}_{n}(1 \square \underline{23})|$ for every positive $n$ and we already placed those two patterns in the same of the three classes for the considered set of vincular DPs. Using this, we can easily prove the following 
\begin{theorem}\label{th:123less}
If $n>4$, then $|\operatorname{Av}_{n}(1 \square \underline{23})| < |\operatorname{Av}_{n}(1 \square \underline{32})|$.
\end{theorem}
\begin{proof}
We just noted that $|\operatorname{Av}_{n}(1 \square \underline{23})|$ is given by the number $a_{n}$ from Theorem \ref{th:vinc123}, while $|\operatorname{Av}_{n}(1 \square \underline{32})|$ is given by the number $b_{n}$ from equation $\eqref{eq:vinc132}$. By substituting $j = n-i+1$, we get that the sum in the right-hand side of equation $\eqref{eq:vinc132}$ can be written as
\begin{equation}
    \sum\limits_{j=4}^{n-1}((n-j+1)\binom{n-2}{j-3} + \binom{n-1}{j-1})b_{n-j}.
\end{equation} 
Then, in order to obtain this inequality, it suffices to prove that for every $n>4$ and $4 \leq i \leq n-1$:
\begin{equation}\label{ineq:summands}
    \binom{n}{i-1}-1 < (n-i+1)\binom{n-2}{i-3} + \binom{n-1}{i-1}.
\end{equation}
This is equivalent to $\binom{n-1}{i-2} -1 < (n-i+1)\binom{n-2}{i-3}$ or $\binom{n-1}{i-2} -1 < \frac{(n-i+1)(i-2)}{n-1}\binom{n-1}{i-2}$. When $n=5$ and $i=4$, we check directly that the latter holds. When $n>5$, one can easily see that $\frac{(n-i+1)(i-2)}{n-1} > 1$, for $4 \leq i \leq n-1$.
\end{proof}

It remains to investigate Class 3. A well known proof technique in the area of permutation patterns helps to do that.
\begin{theorem}\label{th:vinc13sq2}
For all $n\in\mathbb{Z^{+}}$, $\operatorname{Av}_{n}(\underline{13} \square 2) = \operatorname{Av}_{n}(13 \square 2)$, which implies that $|\operatorname{Av}_{n}(\underline{13} \square 2)| = |\operatorname{Av}_{n}(13 \square 2)|$.
\end{theorem}
\begin{proof}
We will prove that whenever an $n$-permutation contains the pattern $13 \square 2$, then it must contain the pattern $\underline{13} \square 2$. Take an $n$-permutation $\sigma = \sigma_{1}\sigma_{2} \cdots \sigma_{n}$ containing $q = 13 \square 2$ and let $\sigma_{i}\sigma_{j}\sigma_{k}$, $1 \leq i<j<k-1 < n$ be an occurrence of $q$ with the smallest possible distance between the $1$ and the $3$, i.e., $d = j-i$ is the smallest possible for such an occurrence. If $d = 1$, then $\sigma_{i}\sigma_{j}\sigma_{k}$ would be an occurrence of $\underline{13} \square 2$ and we are done. Assume that $d>1$ and then consider the value of $\sigma_{i+1}$. If $\sigma_{i+1}<\sigma_{k}$, then $\sigma_{i+1}\sigma_{j}\sigma_{k}$ would be a $q$-occurrence with $j-(i+1) = d-1 < d$. On the other hand, if $\sigma_{i+1}>\sigma_{k}$, then $\sigma_{i}\sigma_{i+1}\sigma_{k}$ would be a $q$-occurrence with $(i+1)-i = 1 < d$, which is again a contradiction.
\end{proof}
The theorem that we just proved and the fact that $|\operatorname{Av}_{n}(12\square 3)| = |\operatorname{Av}_{n}(13\square 2)|$ (see \cite{thesis} and subsection \ref{sec:polygons}) imply that $|\operatorname{Av}_{n}(\underline{13} \square 2)|$ is given by the right-hand side of equation \eqref{eq:firro123} and sequence A049124 in OEIS. It turns out that the patterns in the corresponding Class 3 of Table 1 have the smallest avoiding sets out of the 3 classes.

\begin{theorem}\label{th:123greater}
For all $n \geq 5$, $|\operatorname{Av}_{n}(\underline{12} \square 3)| > |\operatorname{Av}_{n}(\underline{13} \square 2)|$.
\end{theorem}
To establish this fact, we will first need a few additional definitions. For a given pattern $q$, let $\operatorname{Av}_{i_{1},i_{2},\dots ,i_{k};n}(q)$ be the set of $q$-avoiders of size $n$ beginning with $i_{1},i_{2},\dots ,i_{k}$ and let $av_{i_{1},i_{2},\dots ,i_{k};n}(q)$ denotes $|\operatorname{Av}_{i_{1},i_{2},\dots ,i_{k};n}(q)|$. Moreover, let $av_{n}(q) \coloneqq|\operatorname{Av}_{n}(q)|$. Let us first prove the following simple lemma.
\begin{lemma}\label{lemma:key}
If $1\leq i \leq n-1$ and $n\geq 4$, then 
\begin{equation*}
av_{i;n}(\underline{13}\square 2) \leq av_{n;n}(\underline{13}\square 2) = av_{n-1}(\underline{13}\square 2).
\end{equation*}
Moreover, if $1\leq i \leq n-2$, then the inequality is strict, i.e., 
\begin{equation*}
av_{i;n}(\underline{13}\square 2) < av_{n;n}(\underline{13}\square 2).
\end{equation*}
\end{lemma}
\begin{proof}
For every $\pi\in \operatorname{Av}_{n-1}(\underline{13}\square 2)$, we have that $n\pi \in \operatorname{Av}_{n;n}(\underline{13}\square 2)$, since $n$ cannot participate in any occurrences of $\underline{13}\square 2$, being at first position. Conversely, for every $n\sigma \in \operatorname{Av}_{n;n}(\underline{13}\square 2)$, one have that $\sigma\in \operatorname{Av}_{n-1}(\underline{13}\square 2)$. Thus, $av_{n;n}(\underline{13}\square 2) = av_{n-1}(\underline{13}\square 2)$. In addition, for every $1\leq i \leq n-1$ and $i\sigma \in \operatorname{Av}_{i;n}(\underline{13}\square 2)$, we have $fl(\sigma)\in \operatorname{Av}_{n-1}(\underline{13}\square 2)$, which implies that $av_{i;n}(\underline{13}\square 2)\leq av_{n-1}(\underline{13}\square 2)$.

Since $n\geq 4$, when $1\leq i \leq n-2$, then we have at least one $n$-permutation $\pi = in\sigma'$, beginning with $i$, where $n\sigma' = \sigma$ is such that $fl(\sigma)\in \operatorname{Av}_{n-1}(\underline{13}\square 2)$ and $i$ is obviously part of an $\underline{13}\square 2$-occurrence. An example is $\pi = ina\cdots (n-1)$ for any $a\in [n]$, where $a\neq i,n,(n-1)$. Therefore, $fl(\sigma)\in \operatorname{Av}_{n-1}(\underline{13}\square 2)$, but $\pi = i\sigma \notin \operatorname{Av}_{i,n}(\underline{13}\square 2)$. 
\end{proof}
We will need this lemma together with a few other definitions. Given a permutation (pattern) $\sigma$, let $C_{n}(\sigma) = S_{n} \setminus \operatorname{Av}_{n}(\sigma)$ be the permutations of $S_{n}$ containing $\sigma$. Then, let 
\begin{center}
$U_{n} \coloneqq C_{n}(\underline{12} \square 3)\cap \operatorname{Av}_{n}(\underline{13} \square 2)$    
\end{center}
and 
\begin{center}
$V_{n} \coloneqq \operatorname{Av}_{n}(\underline{12} \square 3)\cap C_{n}(\underline{13} \square 2)$,  
\end{center}
with $u_{n} \coloneqq |U_{n}|$ and $v_{n} \coloneqq |V_{n}|$. In addition, let us denote with $U_{i_{1},i_{2},\dots ,i_{k};n}$ (resp. $V_{i_{1},i_{2},\dots ,i_{k};n}$) the set of permutations in $U_{n}$ (resp. in $V_{n}$) beginning with $i_{1}i_{2}\dots i_{k}$.
Furthermore, let $u_{i_{1},i_{2},\dots ,i_{k};n} \coloneqq |U_{i_{1},i_{2},\dots ,i_{k};n}|$ and $v_{i_{1},i_{2},\dots ,i_{k};n} \coloneqq |V_{i_{1},i_{2},\dots ,i_{k};n}|$. Now, we will prove the following 
\begin{lemma}\label{lemma:almost}
For each $n\geq 4$ and $1\leq i \leq n$,
\begin{center}
    $u_{i;n} \leq v_{i;n}$.
\end{center}
\end{lemma}
\begin{proof}
Note that the statement implies $u_{n}\leq v_{n}$ and $|C_{n}(\underline{12} \square 3)|\leq |C_{n}(\underline{13} \square 2)|$ (resp. $|\operatorname{Av}_{n}(\underline{12} \square 3)|\\{\geq} |\operatorname{Av}_{n}(\underline{13} \square 2)|$), for each $n\geq 4$. Indeed, if $T_{n} = C_{n}(\underline{12} \square 3)\cap C_{n}(\underline{13} \square 2)$, then $C_{n}(\underline{12} \square 3) = U_{n}\cup T_{n}$ and $C_{n}(\underline{13} \square 2) = V_{n}\cup T_{n}$. Thus, $u_{n}\leq v_{n}$ implies $|C_{n}(\underline{12} \square 3)|\leq |C_{n}(\underline{13} \square 2)|$. We will proceed by induction on $n$. One can directly check that $u_{i;4} \leq v_{i;4}$ for each $1\leq i \leq 4$. Now assume that $u_{i;n'} \leq v_{i;n'}$, for each $4\leq n' \leq n-1$ and $1\leq i \leq n'$, for a given $n\geq 5$. Consider $u_{i;n}$ and $v_{i;n}$ for $1\leq i \leq n$. If $i = n$, then using the induction hypothesis, we have $u_{n;n} = u_{n-1}\leq v_{n-1} = v_{n;n}$. Similarly, if $i = n-1$, then we have $u_{n-1;n} = u_{n-1}\leq v_{n-1} = v_{n-1;n}$. Now, let $1\leq i \leq n-2$. By the induction hypothesis, $u_{i,i-k;n} = u_{i-k;n-1}\leq  v_{i-k;n-1} = v_{i,i-k;n}$, for each $1\leq k \leq i-1$. It remains to compare the numbers $u_{i,i+k;n}$ and $v_{i,i+k;n}$ for $1\leq k \leq n-i$. Note that when $k\geq 3$, then $u_{i,i+k;n} = 0$, since for these values of $k$, any $n$-permutation beginning with $i(i+k)$ will contain an occurrence of $\underline{13}\square 2$. Similarly, $v_{i,i+k;n} = 0$, when $i+k <n-1$, since for these values of $k$, any $n$-permutation beginning with $i(i+k)$ will contain an occurrence of $\underline{12}\square 3$. We will show that $u_{i,i+1;n}\leq v_{i,n;n}$ and that $u_{i,i+2;n}\leq v_{i,n-1;n}$ which will complete the proof. \\
Consider the sets $U_{i,i+1;n}$ and $V_{i,n;n}$. First, let us look at those $\pi\in U_{i,i+1;n}$ (resp., $\pi \in V_{i,n;n}$) which do not begin with an $\underline{12}\square 3$ occurrence (resp., not with an $\underline{13}\square 2$ occurrence). Then, note that $\pi$ must begin with $(n-2)(n-1)n$ (resp. with $(n-2)n(n-1)$). However, we have 
\begin{equation}\label{eq:notBeg}
    u_{n-2,n-1,n;n} = u_{n-3}\leq v_{n-3} = v_{n-2,n;n-1}
\end{equation}
using the induction hypothesis, again. 

Now, let us look at those $\pi\in U_{i,i+1;n}$ beginning with a $\underline{12}\square 3$ occurrence. Their number is given by
\begin{equation}\label{eq:final1}
av_{i;n-1}(\underline{13}\square 2) - av_{n-2,n-1;n-1}(\underline{13}\square 2).
\end{equation}
 Indeed, after we remove from $\pi$ its first element $i$ and flatten, we obtain an $(n-1)$-avoider of $\underline{13}\square 2$. Conversely, for any permutation $\pi = i\pi_{2}\dots \pi_{n-1}\in \operatorname{Av}_{i;n-1}(\underline{13}\square 2)$, one can increase by $1$ all the elements of $\pi$ greater than or equal to $i$ and then add $i$ at the beginning, to obtain a permutation in $U_{i,i+1;n}$. This permutation will begin with a $\underline{12}\square 3$ occurrence, unless it begins with $(n-2)(n-1)n$, i.e., when $i=n-2$ and when $\pi\in \operatorname{Av}_{n-2,n-1;n-1}(\underline{13}\square 2)$. Therefore, we should subtract $av_{n-2,n-1;n-1}(\underline{13}\square 2)$. Respectively, for the number of permutations $\pi \in V_{i,n;n}$, beginning with an $\underline{13}\square 2$ occurrence, one would have 
 \begin{equation}\label{eq:final2}
 av_{n-1;n-1}(\underline{12}\square 3) - av_{n-1,n-2;n-1}(\underline{12}\square 3).    
 \end{equation}
 It is not difficult to see that $av_{n-2,n-1;n-1}(\underline{13}\square 2) = av_{n-3}(\underline{13}\square 2)$ and that $av_{n-1,n-2;n-1}(\underline{12}\square 3)= av_{n-3}(\underline{12}\square 3)$. Hence, by using expressions \eqref{eq:final1}, \eqref{eq:final2} and equation \eqref{eq:notBeg}, we see that in order to establish that $u_{i,i+1;n}\leq v_{i,n;n}$, it remains to prove the inequality below for each $1\leq i \leq n-2$:
\begin{equation}\label{eq:inAlmost}
    av_{i;n-1}(\underline{13}\square 2) - av_{n-3}(\underline{13}\square 2) \leq av_{n-1;n-1}(\underline{12}\square 3) - av_{n-3}(\underline{12}\square 3).
\end{equation}
By lemma \ref{lemma:key}, we have that $av_{i;n-1}(\underline{13}\square 2)\leq av_{n-1;n-1}(\underline{13}\square 2) = av_{n-2}(\underline{13}\square 2)$. We also have that $av_{n-1;n-1}(\underline{12}\square 3) = av_{n-2}(\underline{12}\square 3)$. Thus, it suffices to prove that
\begin{equation}\label{eq:tmp}
    av_{n-2}(\underline{13}\square 2) - av_{n-3}(\underline{13}\square 2) \leq av_{n-2}(\underline{12}\square 3) - av_{n-3}(\underline{12}\square 3).
\end{equation}
Using that $av_{n-3}(q) = av_{n-2;n-2}(q)$ for any of the patterns $q = \underline{12}\square 3$ or $q = \underline{13}\square 2$, as well as the relation 
\begin{equation}
av_{i;n-2}(\underline{13}\square 2)\leq av_{i;n-2}(\underline{12}\square 3) \Longleftrightarrow v_{i;n-2}\geq u_{i;n-2},    
\end{equation}
we see that equation \eqref{eq:tmp} is equivalent to 
\begin{equation}\label{eq:final}
    \sum\limits_{i=1}^{n-3} v_{i;n-2} \geq \sum\limits_{i=1}^{n-3} u_{i;n-2},
\end{equation}
which follows directly, because by the induction hypothesis $u_{i;n-2}\leq v_{i;n-2}$, $\forall 1\leq i \leq n-3$. From equations \eqref{eq:notBeg} and \eqref{eq:final}, we conclude that $u_{i,i+1;n}\leq v_{i,n;n}$.

One could establish that $u_{i,i+2;n}\leq v_{i,n-1;n}$ in almost the same way, by first noticing that $U_{i,i+2;n} = U_{i,i+2,i+1;n}$ and that $V_{i,n-1;n} = V_{i,n-1,n;n}$ since the permutations in $U_{i,i+2;n}$ (resp. in $V_{i,n-1;n}$) do not have a $\underline{13}\square 2$ (resp. a $\underline{12}\square 3$) occurrence. Then, the only thing that remains is to consider the cases $i=n-2$ and $i\neq n-2$ and to use the induction hypothesis and lemma \ref{lemma:key}. In particular, if $i=n-2$, then 
\begin{equation}
    u_{n-2,n,n-1;n} = u_{n-3}\leq v_{n-3} = v_{n-2,n-1,n;n}.
\end{equation}
If $i\neq n-2$, then $\pi\in U_{i,i+2,i+1;n}$ (resp. in $V_{i,n-1,n;n}$) begins with an $\underline{12}\square 3$ (resp. an $\underline{13}\square 2$) occurrence and 
\begin{equation}
    u_{i,i+2,i+1;n} = av_{i;n-2}(\underline{13}\square 2) \leq av_{n-2;n-2}(\underline{13}\square 2)
    \end{equation}
by lemma \ref{lemma:key}. In addition,
    \begin{equation}
    av_{n-2;n-2}(\underline{13}\square 2) \leq av_{n-2;n-2}(\underline{12}\square 3) = v_{i,n-1,n;n}
\end{equation}
by the induction hypothesis.
\end{proof}
As we have pointed out, lemma \ref{lemma:almost} implies that $|\operatorname{Av}_{n}(\underline{12} \square 3)|\geq |\operatorname{Av}_{n}(\underline{13} \square 2)|$, for each $n\geq 4$. In order to obtain a proof of theorem \ref{th:123greater}, we should just use the second part of lemma \ref{lemma:key} to see that inequality \eqref{eq:inAlmost} is strict when $n-1\geq 4$, i.e., when $n\geq 5$.

We are now ready to formulate an interesting conclusion. The last Theorem \ref{th:123greater} together with Theorem \ref{th:123less}, the result on consecutive patterns of Elizalde \cite{Elizalde2003} and the corollary of the result of Hopkins (Theorem \ref{th:123IsLargest}) imply the following:
\begin{corollary}\label{cor:venn}
Consider the set of distant patterns 
\begin{equation*}
    X = \{1\square\underline{23}, \underline{12}\square 3, 1\square 2 \square 3, \underline{123}\}.
\end{equation*} 
Take any pattern $p\in X$ and switch the places of the letters $2$ and $3$ to get a pattern $p'$ in $Y = \{1\square\underline{32},\underline{13}\square 2,1\square 3 \square 2, \underline{132}\}$. We have that $\operatorname{Av}_{n}(p)>\operatorname{Av}_{n}(p')$ for all $n>5$, $p\in X$ and the corresponding $p'\in Y$, except for $1\square \underline{23}$ which is avoided by fewer permutations of size $n$, compared to its counterpart $1\square \underline{32}$.
\end{corollary}


 \begin{figure}[h!]
      \centering
        \begin{tikzpicture}[scale=0.85]
  \draw (1,0) circle (1.5);
  \node[above] at (0.6,0.2) {$C_{n}(\underline{123})$}; 
  \filldraw[fill = {rgb:black,1;white,10}] (3,0) circle (1.5);
  \node[above] at (3.3,0.2) {\resizebox{0.1\hsize}{!}{$C_{n}(1\square\underline{23})$}}; 
  \draw (1,-1.8) circle (1.5);
  \node[below] at (0.6,-2) {\resizebox{0.1\hsize}{!}{$C_{n}(\underline{12}\square 3)$}}; 
  \draw (3,-1.8) circle (1.5);
  \node[below] at (3.4,-2) {\resizebox{0.1\hsize}{!}{$C_{n}(1\square 2 \square 3)$}}; 

    \draw (6.5,0) circle (1.5);
  \node[above] at (6.1,0.2) {$C_{n}(\underline{132})$}; 
  \filldraw[fill = {rgb:black,1;white,10}] (8.5,0) circle (1.5);
  \node[above] at (8.8,0.2) {\resizebox{0.1\hsize}{!}{$C_{n}(1\square\underline{32})$}}; 
  \draw (6.5,-1.8) circle (1.5);
  \node[below] at (6.1,-2) {\resizebox{0.1\hsize}{!}{$C_{n}(\underline{13}\square 2)$}}; 
  \draw (8.5,-1.8) circle (1.5);
  \node[below] at (8.9,-2) {\resizebox{0.1\hsize}{!}{$C_{n}(1\square 3 \square 2)$}}; 
    \end{tikzpicture}
  \caption{Venn diagrams for the $n$-permutations containing $123$ and $132$}
  \label{fig:venn}
\end{figure}

Figure \ref{fig:venn} depicts the sets of permutations containing each of the patterns in $X$ and $Y$ as a Venn diagram. Corollary \ref{cor:venn} is somewhat surprising since each occurrence of the classical pattern $123$ (resp.$132$) is an occurrence of a pattern in $X$ (resp. $Y$) and as it was shown in \cite{simion}, $|\operatorname{Av}_{n}(123)| = |\operatorname{Av}_{n}(132)|$. Thus the total "area" of the union of the four sets on the left is the same as the total "area" of the union of the four sets on the right. However, each of the three unmarked sets on the left contains fewer elements than its counterpart on the right. 

\subsection{Consecutive distant patterns} \label{sec:consecutiveDP} 
Recall that when all the constraints for the  gap sizes in a distant pattern are tight, then we call these patterns consecutive distant patterns and we underline the whole pattern to denote that. Considering POGP, Kitaev mentioned in the introduction of \cite{Kit07} that $\operatorname{Av}_{n}(\underline{1 \square 2}) = \binom{n}{\lfloor\frac{n}{2}\rfloor}$. Indeed, we may use that the letters in the odd and the even positions of a permutation avoiding this pattern do not affect each other. Thus we can choose the letters in odd positions in $\binom{n}{\lfloor\frac{n}{2}\rfloor}$ ways, and we must arrange them in decreasing order. We then must arrange the letters in even positions in decreasing order, too. Using the same reasoning one can easily find, for example $\operatorname{Av}_{n}(\underline{1\square^{2}2})$ or $\operatorname{Av}_{n}(\underline{1 \square 2 \square 3})$. This can be further generalized by the fact given below. Recall that if $q = q_{1}q_{2}\cdots q_{k}$ is a classical pattern of size $k$, then $\underline{q} = \underline{q_{1}q_{2}\cdots q_{k}}$ is the corresponding consecutive pattern. We also use  $\underline{\operatorname{dist}_{r}(q)}$ to denote the corresponding consecutive distant pattern $\underline{q_{1}\square^{r}q_{2}\square^{r}\cdots \square^{r}q_{k}}$.
\begin{theorem}\cite[Theorem 11]{Kit05Main}
For a given classical pattern $q$ of size $k$, given distance $r\geq 0$ and a natural $n$, denote $l = \lfloor\frac{n}{r+1}\rfloor$. Set $u \coloneqq n \mod (r+1) \in [0,r]$. Then
\begin{equation} \label{eq:APs}
|\operatorname{Av}_{n}(\underline{\operatorname{dist}_{r}(q)})|=\frac{n!}{(l!)^{r+1-u}((l+1)!)^{u}}|A_{l}(\underline{q})|^{r+1-u}|A_{l+1}(q)|^{u}.    
\end{equation}
\end{theorem}
This gives us a formula for the size of the set of permutations avoiding any consecutive distant pattern, knowing the size of the avoidance set for the corresponding classical consecutive pattern. Corollaries of this simple fact were previously stated in \cite{thesis, MansourPolygons}. We state another simple corollary here, which shows a surprising relationship between the former fact and avoidance of arithmetic progressions in permutations.
\begin{theorem}
The number of permutations of size $n$ avoiding arithmetic progressions of length $k>1$ and difference $r>0$ is $|\operatorname{Av}_{n}(\underline{\operatorname{dist}_{r}(12\cdots k)})|$, which can be obtained using equation \eqref{eq:APs}.
\end{theorem}
\begin{proof}
Consider $\pi\in S_{n}$, containing an arithmetic progression (AP) $\pi_{i_{1}}\pi_{i_{2}}\cdots \pi_{i_{k}}$ of size $k$ and difference $r>0$. I.e., we have $\pi_{i_{1}} = x$, $\pi_{i_{2}} = x+r$,$\cdots$, $\pi_{i_{k}} = x+(k-1)r$ for some $x,r\in [n]$ with $i_{1}<i_{2}<\cdots <i_{k}$. Then in the inverse permutation $\pi^{-1}$, $i_{1}i_{2}\cdots i_{k}$ will be an occurrence of the distant pattern $\underline{\operatorname{dist}_{r}(12\cdots k)}$ since $\pi^{-1}(x) = i_{1}$, $\pi^{-1}(x+r) = i_{2}$, $\cdots$, $\pi^{-1}(x+(k-1)r) = i_{k}$. Conversely, if $\pi\in S_{n}$ contains $\underline{\operatorname{dist}_{r}(12\cdots k)}$, then $\pi^{-1}$ contains an AP of length $k$ and difference $r>0$. Therefore, the number of permutations of $[n]$ containing APs of length $k$ and difference $r>0$ equals the number of permutations of $[n]$ containing $\underline{\operatorname{dist}_{r}(12\cdots k)}$. This implies the same for the set of avoiders, i.e., what we aim to prove.
\end{proof}


\section{Interpretations of other results}\label{sec:conj}
Here, we will demonstrate that DPs can be very useful when interpreting already known results (including ones obtained with a computer). One previous work that gives several conjectures about the enumeration of pattern-avoiding classes containing many size four patterns is the work of Kuszmaul \cite{Kuszmaul}. He listed ten conjectures about simultaneous pattern avoidance of many size four patterns. One can find brief solutions, using both computer programs and manual work, to many of these conjectures in the two articles of Mansour and Schork \cite{MansSchork8,MansSchork67}.

 Below, we give direct solutions to two of the conjectures without using a computer. To do that, we interpret the respective big set of size four patterns as a smaller set of both classical and distant patterns. Our approach is similar to the technique introduced in \cite{MansCell}.
\begin{theorem} (conjectured in \cite{Kuszmaul}, p.20, sequence 6)
\label{th:conj6}
The generating function of
\begin{equation*}
    |\operatorname{Av}_{n}(2431,2143,3142,4132,1432,1342,1324,1423,1243)|
\end{equation*} is given by $C+x^{3}C$, where $C$ is the generating function for the Catalan numbers.
\end{theorem}
\begin{proof}
Note that the set of patterns above can be written as 
\begin{equation*}
\Pi = \{\square132, 132\square, 1342\}   
\end{equation*}
When $n=1,2,3$, $|\operatorname{Av}_{n}(\Pi)| = 1,2,6$ respectively and these are indeed the first three coefficients of $C+x^{3}C$. Consider values $n\geq 4$. If $\sigma\in \operatorname{Av}_{n}(132)$, then $\sigma\in \operatorname{Av}_{n}(\Pi)$. As we know, $|\operatorname{Av}_{n}(132)|$ has generating function $C$ \cite{knuth}.
It remains to find the generating function for those $\sigma$ containing $132$, but avoiding $\Pi$. Take one such $\sigma = \sigma_{1}\sigma_{2}\cdots \sigma_{n}$ and notice that any occurrence of $132$ in $\sigma$ must have $\sigma_{1}$ as its first letter and $\sigma_{n}$ as its last letter. Otherwise, given that $n\geq 4$, an occurrence of at least one of the patterns $132\square$ or $\square 132$ will be present. Now, let $\sigma_{k} = n$ be the biggest element of $\sigma$. Clearly, $\sigma_{1}\sigma_{k}\sigma_{n}$ must be an occurrence of $132$. If not, then this biggest element must be either at the first or the last position in $\sigma$ and thus $\sigma$ would not contain any $132$-occurrences that either begin at $\sigma_{1}$ or end at $\sigma_n$. In Figure \ref{fig:conj6} are shown three black points representing $\sigma_1$, $\sigma_k$ and $\sigma_{n}$, as well as three segments of the diagram of $\sigma$ denoted with $A$, $B$ and $C$ and defined below. We further show that $\sigma$ will not contain any elements in these three segments. Here is why:
\begin{itemize}
    \item $A$ is empty \\
    There is no element $x$ among $\sigma_{2},\sigma_{3},\cdots ,\sigma_{k-1}$, such that $x<\sigma_{n}$. If there is such $x$, then $x\sigma_{k}\sigma_{n}$ would be a forbidden occurrence of $132$.
    \item $B$ is empty \\
    There is no element $x$ among $\sigma_{k+1},\sigma_{k+2},\cdots ,\sigma_{n-1}$, such that $\sigma_{1}<x<\sigma_{k}$. If there is such $x$, then $\sigma_{1}\sigma_{k}x$ would be a forbidden occurrence of $132$.
    \item $C$ is empty \\
    There is no element $x$ among $\sigma_{1},\sigma_{2},\cdots ,\sigma_{k-1}$, such that $\sigma_{n}<x<\sigma_{k}$. If there is such $x$, then $\sigma_{1}x\sigma_{k}\sigma_{n}$ would be an occurrence of $1342$.
\end{itemize}
Therefore, the biggest element $\sigma_{k}$ in $\sigma$ must be at position $2$, i.e., $k=2$ and the only non-empty segment could be the one denoted by $\alpha$ in figure \ref{fig:conj6}. In other words, we must have $\sigma=(n-2)n\alpha(n-1)$, for some permutation $\alpha\in \operatorname{Av}(132)$. Otherwise, an occurrence of $132$, such that $\sigma_{1}= n-2$ is not part of it, would be formed. Conversely, for any $\alpha\in \operatorname{Av}(132)$, $\sigma = (n-2)n\alpha(n-1)$ belongs to $\operatorname{Av}(\Pi)$.

\begin{figure}
            \centering
        \begin{tikzpicture}[scale=0.8]
    \node[below] at (-3,2.75) {$\sigma_{1}$}; 
    \filldraw [black] (-3,3) circle (2.5pt);
    \node[above] at (0,7.25) {$\sigma_{k}$}; 
    \filldraw [black] (0,7) circle (2.5pt);
    \node[below] at (3,4.75) {$\sigma_{n}$}; 
    \filldraw [black] (3,5) circle (2.5pt);
    
    \draw[dotted] (-3,3) -- (3,3);
    \draw[dotted] (-3,7) -- (3,7);
    \draw[dotted] (-3,5) -- (3,5);
    \draw[dotted] (0,7) -- (0,1.5);
    
    \draw[draw=gray] (-2.5,1.5) rectangle ++(2,3.25);
    \draw (-2.5,1.5) -- (-0.5,4.75);
    \draw (-0.5,1.5) -- (-2.5,4.75);
    \node at (-1.5,4.4) {\large{A}}; 
    
    \draw[draw=gray] (-2.5,5.25) rectangle ++(2,1.5);
    \draw (-2.5,5.25) -- (-0.5,6.75);
    \draw (-0.5,5.25) -- (-2.5,6.75);
    \node at (-1.5,6.4) {\large{C}}; 
    
    \draw[draw=gray] (0.5,3.25) rectangle ++(2,3.5);
    \draw (0.5,3.25) -- (2.5,6.75);
    \draw (2.5,3.25) -- (0.5,6.75);
    \node at (1.5,6.35) {\large{B}};  
    
    \draw[draw=gray] (0.5,1.5) rectangle ++(2,1.25);
    \node at (1.5,2.5) {\large{$\alpha$}};

    \end{tikzpicture}
         \caption{Decomposition for $\sigma\in \operatorname{Av}( \Pi )$ from the proof of Theorem \ref{th:conj6}}
        \label{fig:conj6}
        \end{figure}


Then, we get $x^{3}C$ for the generating function of the permutations in $\operatorname{Av}(\Pi)$ containing $132$ and therefore we will have $C+x^{3}C$ for the generating function of $\operatorname{Av}(\Pi)$, since $C$ is the generating function for $\operatorname{Av}(132)$. 
\end{proof}
As we know, $C = 1+xC^{2}$, so we can write
\begin{equation*}
C+x^{3}C = C+x^{3}(1+xC^{2}) = x^{3} + C(1+x^{4}C)
\end{equation*}
and this indeed corresponds to sequence A071742 given by $C(1+x^{4}C)$, as reported in \cite{Kuszmaul}, with the subtle difference that for $n=3$, we have one extra permutation in $\operatorname{Av}_{3}(\Pi)$, namely $132$. The same structure for the decomposition of the permutations in $\operatorname{Av}(\Pi)$ was also found recently with a computer by Bean et al. who used a particular algorithm called the \emph{Struct algorithm} \cite{bean}. As we saw, rewriting the problem in terms of distant patterns helped us to prove the result directly and to give an interpretation of the already discovered decomposition. 

Below, we will give a direct proof to another former conjecture listed in \cite{Kuszmaul}.
\begin{theorem} (conjectured in \cite{Kuszmaul}, p.19, sequence 5)
\label{th:conj5}
The generating function of
\begin{equation*}
    |\operatorname{Av}_{n}(2431,2413,3142,4132,1432,1342,1324,1423)|
\end{equation*} is given by $C(1+x^{3}C)$, where $C$ is the generating function for the Catalan numbers.
\end{theorem}
\begin{proof}
Note that the set of permutations above can be written as 
\begin{equation*}
\Pi = \{13\square2, 1324, 2431, 3142, 4132\}.   
\end{equation*}
If $\sigma$ has no occurrences of $132$ at all, then obviously $\sigma\in \operatorname{Av}(\Pi)$ and the generating function for these permutations is $C$. Let us consider those $\sigma$ that have some occurrences of $132$ and are in $\operatorname{Av}(\Pi)$. The set $\Pi$ contains $13\square2$ thus all the occurrences of $132$ in $\sigma$, are occurrences of $1\underline{32}$. Take the occurrence $\sigma_{i}\sigma_{j}\sigma_{j+1}$ of $1\underline{32}$ that ends at the largest possible position, i.e., with $j$ maximal. Denote by $\alpha$ the segment in $\sigma$ of largest possible size that ends at $\sigma_{i}$ and such that $\alpha<\sigma_{j+1}<\sigma_{j}$. Let us first consider $\sigma' = \sigma_{j+2}\sigma_{j+3}\cdots \sigma_{n}$. We will show that $\sigma'$ is the empty permutation, i.e., $n=j+1$ and the segments $A$,$B$ and $C$, defined below and shown at Figure \ref{fig:conj5} are empty:
\begin{itemize}
    \item $A$ is empty \\
    There is no element $x$ in $\sigma'$, such that $x < \sigma_{j}$ and $x \nless \alpha$. If there is such $x$, then $\sigma_{i}\sigma_{j}x$ would be an occurrence of $132$ that is not an $1\underline{32}-$occurrence.
    \item $B$ is empty \\
    There is no element $x$ in $\sigma'$, such that $x<\alpha$. If there is such $x$, then $\sigma_{i}\sigma_{j}\sigma_{j+1}x$ would be an occurrence of 2431 which is not allowed.
    \item $C$ is empty \\
    There is no element $x$ in $\sigma'$, such that $x>\sigma_{j}$. If there is such $x$, then $\sigma_{i}\sigma_{j}\sigma_{j+1}x$ would be an occurrence of 1324 which is not allowed.
\end{itemize}
Next, let us consider the segment $\sigma'' = \sigma_{i+1}\cdots \sigma_{j-1}$. We will show that $\sigma''$ is the empty permutation, i.e., the segment $D$, shown at Figure \ref{fig:conj5}, is empty and $i=j-1$:
\begin{itemize}
    \item $D$ is empty \\
    There is no element $x$ in $\sigma''$, such that $x>\sigma_{j+1}$. If there is such $x$, then $\sigma_{i}x\sigma_{j+1}$ would be an occurrence of $132$ that is not an $1\underline{32}-$occurrence.
\end{itemize}

\begin{figure}
            \centering
        \begin{tikzpicture}[scale=0.75]
    \node at (0.25,7.25) {$\sigma_{j}$}; 
    \filldraw [black] (0,7) circle (2.5pt);
    \node at (0.75,6.25) {$\sigma_{j+1}$}; 
    \filldraw [black] (1,6) circle (2.5pt);
   
    \draw[dotted] (0,10) -- (0,2);
    \draw[dotted] (1,10) -- (1,2);
    \draw[dotted] (-6,7) -- (3,7);
    \draw[dotted] (-6,6) -- (1,6);
    
    \draw[draw=gray] (1.5,7.25) rectangle ++(1.5,2.25);
    \draw (1.5,7.25) -- (3,9.5);
    \draw (3,7.25) -- (1.5,9.5);
    \node at (2.25,9.25) {\large{C}}; 
    
    \draw[draw=gray] (1.5,4) rectangle ++(1.5,2.75);
    \draw (1.5,4) -- (3,6.75);
    \draw (3,4) -- (1.5,6.75);
    \node at (2.25,6.5) {\large{A}}; 
    
    \draw[dotted] (-4,3.75) -- (3,3.75);
    
    \draw[draw=gray] (1.5,1.25) rectangle ++(1.5,2.25);
    \draw (1.5,1.25) -- (3,3.5);
    \draw (3,1.25) -- (1.5,3.5);
    \node at (2.25,3.25) {\large{B}}; 
    
    \draw[draw=gray] (-4,4) rectangle ++(2,1.75);
    \node at (-3,5) {\large{$\alpha$}}; 
    
    \node at (-1.65,5) {$\sigma_{i}$}; 
    \filldraw [black] (-2,5) circle (2.5pt);
    
    \draw[draw=gray] (-5.75,6.075) rectangle ++(1.5,0.75);
    \draw (-5.75,6.075) -- (-4.25,6.825);
    \draw (-4.25,6.075) -- (-5.75,6.825);
    \node at (-5,6.4) {E}; 
    
    \draw[draw=gray] (-5.75,7.25) rectangle ++(1.5,0.75);
    \draw (-5.75,7.25) -- (-4.25,8);
    \draw (-4.25,7.25) -- (-5.75,8);
    \node at (-5,7.6) {F}; 
    
        \draw[draw=gray] (-1.75,6.075) rectangle ++(1.5,2.075);
    \draw (-1.75,6.075) -- (-0.25,8.15);
    \draw (-0.25,6.075) -- (-1.75,8.15);
    \node at (-1,7.75) {\large{D}};

    \end{tikzpicture}
    \caption{Decomposition for $\sigma\in \operatorname{Av}( \Pi )$ from the proof of Theorem \ref{th:conj5}.}
    \label{fig:conj5}
        \end{figure}


Finally, consider the segment $\sigma'''$ that is the part of $\sigma$ in front of $\alpha$. We will show that $\sigma'''$ is the empty permutation, i.e., the segments E and F, shown at Figure \ref{fig:conj5} are empty:
\begin{itemize}
    \item E is empty \\
    There is no element $x$ in $\sigma'''$, such that $x > \sigma_{j+1}$ and $x < \sigma_{j}$. If there is such $x$, then $x\sigma_{i}\sigma_{j}\sigma_{j+1}$ would be an occurrence of $3142$, which is forbidden.
    \item F is empty \\
    There is no element $x$ in $\sigma'''$, such that $x > \sigma_{j}$. If there is such $x$, then $x\sigma_{i}\sigma_{j}\sigma_{j+1}$ would be an occurrence of $4132$, which is forbidden.
\end{itemize}
Therefore, we must have $\sigma = \alpha n(n-1)$, where $\alpha$ is non-empty and $\alpha \in \operatorname{Av}(132)$. The latter holds since if $\alpha$ contains 132 then after appending $\sigma_{j}$ at the end, we will get an occurrence of 1324.
Conversely, one may readily check that for each non-empty $\alpha\in \operatorname{Av}(132)$, $\sigma = \alpha n(n-1)$ would be a permutation in $\operatorname{Av}(\Pi)$.
Thus, the generating function of the number of permutations in $\operatorname{Av}(\Pi)$ is $C + x^{2}(C-1)$, since the generating function of such non-empty $\alpha\in \operatorname{Av}(132)$ is $C-1$. Furthermore, we have
\begin{equation*}
    C + x^{2}(C-1) = C + x^{2}xC^{2} = C(1+x^{3}C).
\end{equation*}
\end{proof}
The sequence of the coefficients for this generating function is given by A071726 in OEIS.

\section{Stanley--Wilf type conjectures for DPs}\label{sec:SW}
A popular former conjecture on the classical permutation patterns formulated independently by Stanley and Wilf states that for any given classical pattern $q$, there exists a constant $c_{q}$, such that $\sqrt[n]{|\operatorname{Av}_{n}(q)|} < c_{q}$, when $n\to \infty$. In 1999, Arratia (\cite{arratia}) observed that this is equivalent to the existence of the limit $\lim\limits_{n\to \infty}\sqrt[n]{|\operatorname{Av}_{n}(q)|}$. The conjecture was resolved in 2004 by Markos and Tardos (\cite{MarcusTardos}) who actually proved a conjecture of Füredi and Hajnal, which had been shown earlier to imply the Stanley–Wilf conjecture.

In Theorem \ref{th:simultAv}, we saw that the avoidance of every distant pattern is equivalent to simultaneous avoidance of several classical patterns. The Stanley--Wilf conjecture is true for any of these classical patterns. Thus we will have that $\sqrt[n]{|\operatorname{Av}_{n}(q)|} < const$, for any distant pattern $q$, when $n\to \infty$. Arratia's observation that $|\operatorname{Av}_{m+n}(q)|\geq |\operatorname{Av}_{m}(q)|.|\operatorname{Av}_{n}(q)|$, also holds for distant patterns, if the considered distant pattern does not start with a square. Thus for those kind of distant patterns, we can rely on the Fekete's lemma on subadditive sequences, exactly as Arratia did, to obtain that $\sqrt[n]{|\operatorname{Av}_{n}(dp)|}$ exists. As for the DPs beginning with $r>0$ number of squares, we can use Theorem \ref{th:sqBeginEnd} to write \begin{equation*}
\sqrt[n]{|\operatorname{Av}_{n}(\square^{r}q)|} = \sqrt[n]{n^{(r)}|\operatorname{Av}_{n-r}(q)|} = (n^{(r)})^{\frac{1}{n}}|\operatorname{Av}_{n-r}(q)|^{\frac{1}{n-r}\frac{n-r}{n}} \underset{n\rightarrow \infty}{\longrightarrow}  c_{q},  
\end{equation*}
where $q$ is a distant pattern which does not start with a square and \\ $\sqrt[n]{|\operatorname{Av}_{n}(q)|} \longrightarrow c_{q}$. This yields the following Stanley--Wilf type result for DPs.
\begin{theorem}
For any distant pattern $q$, there exists a constant $c>0$, such that
\begin{equation}
    \sqrt[n]{|\operatorname{Av}_{n}(q)|} \underset{n\rightarrow \infty}{\longrightarrow} c_{q}.
\end{equation}
\end{theorem}
An interesting continuation might be to consider avoidance of $\operatorname{dist}_{r}(q)$, for a classical pattern $q$ and size of $r$ that increases with $n$. Obviously, if $r\geq n-1$, then $|\operatorname{Av}_{n}(\operatorname{dist}_{r}(q))| = n!$ for any pattern $q$. However, one may ask what will happen if $r$ is a fixed positive fraction of $n$? Is it true that when $n$ is growing, the number of permutations avoiding the corresponding series of DPs is still always converging to $c^{n}$, for some constant $c$? Below, we show that this is not the case, by using Theorem \ref{th:bijection}.

\begin{theorem}
 \label{conj:1}
It is not true that for any given classical pattern $q$, there exist constants $c > 0$ and $0 <c_{1} < 1$, such that 
\begin{equation} \label{eq:SWlimitFractionalR}
\sqrt[n]{|\operatorname{Av}_{n}(\operatorname{dist}_{r}(q))|}\overset{n\rightarrow \infty}{\underset{r= \lfloor c_{1}n \rfloor}{\longrightarrow}} c
\end{equation}
\end{theorem}

\begin{proof}
Consider the classical pattern $q = 12$. By Theorem \ref{th:bijection}, the number of $n$-permutations avoiding $1\square^{r}2$, for any $r \geq 1$, will be the same as the number of $n$-permutations for which in each of their cycles, any two elements differ by at most $r$. Denote this set of permutations with $S_{n}^{r}$. Furthermore, let $C_{n}^{r}$ be the set of permutations in $S_{n}$ for which each cycle is of length exactly $r$, except possibly $1$ cycle of smaller length, if $r$ does not divide $n$, and where each cycle is consisted of consecutive elements. Therefore, since $C_{n}^{r}\subseteq S_{n}^{r}$, we can use that $|C_{n}^{r}| \leq |S_{n}^{r}| = |\operatorname{Av}_{n}(1\square^{r}2)|$. In addition, we have the obvious bound $|C_{n}^{r}| \geq ((r-1)!)^{\lfloor \frac{n}{r} \rfloor}$. Thus for any given $0 < c_{1} < 1$, if $r = \lfloor c_{1}n \rfloor$, then for big enough values of $n$, we have: 
\begin{equation*}
\begin{split}
\operatorname{Av}_{n}(1\square^{r}2) \geq ((r-1)!)^{\lfloor \frac{n}{r} \rfloor} = ((\lfloor c_{1}n\rfloor -1)!)^{\lfloor \frac{n}{\lfloor c_{1}n \rfloor} \rfloor} = ((\lfloor c_{1}n\rfloor -1)!)^{\frac{1}{c_{1}}} \geq \\
((\frac{c_{1}}{2}n)!)^{\frac{1}{c_{1}}} \geq ((\frac{c_{1}n}{2e})^{\frac{c_{1}}{2}n})^{\frac{1}{c_{1}}} = (\frac{c_{1}n}{2e})^{\frac{n}{2}} = \sqrt{Cn^{n}} = \Omega(C^{n}),  
\end{split}    
\end{equation*}
for some constant $C>0$. In the last equation, we used the Stirling approximation.  
\end{proof}
The latter fact motivates us to consider the following conjecture.
\begin{conjecture}\label{conj:SWfracLimit}
For any given classical pattern $q$ and for every $0<c_{1}<1$, there exists $0<w<1$, such that:
\begin{equation}
    \lim_{n\to\infty}(\frac{|\operatorname{Av}_{n}(\operatorname{dist}_{\lfloor c_{1}n\rfloor}(q))|}{n!})^{\frac{1}{n}} = w
\end{equation}
\end{conjecture}
The approach of Elizalde (\cite[Section 4]{Elizalde2006}) for consecutive patterns might be useful when one tries to prove the latter conjecture, even though this approach cannot be applied directly. Here, we prove one lemma that might help confirming the conjecture.
\begin{lemma}\label{lemma:SWd}
For any given classical pattern $q\in S_{k}$ and for every $0<c_{1}<1$, there exists $d<1$, such that if $r = \lfloor c_{1}n\rfloor$ and $n \geq k(r+1)$, then
\begin{equation}
    |\operatorname{Av}_{n}(\operatorname{dist}_{r}(q))| < d^{n}n!.
\end{equation}
\end{lemma}
\begin{proof}
Assume that $n \geq k(r+1)$ for some $c_{1}$ and $n$ and let us take an arbitrary permutation $\pi = \pi_{1}\pi_{2}\cdots \pi_{n}$. We can divide the elements of $\pi$ into roughly $\frac{n}{k}$ non-overlapping subsequences of size $k$, such that if $\pi\in \operatorname{Av}_{n}(\operatorname{dist}_{r}(q))$, then neither of these subsequences is order-isomorphic to $q$. We are looking for an upper bound and thus such a necessary condition could help. One way to get such a partition into subsequences is to take
\begin{equation*}
\{\pi_{1}\pi_{r+2}\cdots \pi_{(k-1)r+k}, \pi_{2}\pi_{r+3}\pi_{2r+4}\cdots, \pi_{r+1}\pi_{2r+2}\cdots \},
\end{equation*}
with the first element in every next subsequence being the first not yet used element of $\pi$. Denote this family of subsequences by $\mathbb{F}$ and the event that after a uniform sampling of a permutation $\pi$, no subsequence in  $\mathbb{F}$ is order isomorphic to $q$ by $E_{\mathbb{F},q}$. Since $|\mathbb{F}|\geq (r+1) > 0$, we will have that
\begin{equation}
\mathbb{P}(E_{\mathbb{F},q}) \leq (1-\frac{1}{k!})^{r+1}
\end{equation}
Therefore, if we write $C_{\pi,\operatorname{dist}_{r}(q)}$ for the event that $\pi$ contains the pattern $\operatorname{dist}_{r}(q)$, then \\ $\mathbb{P}(C_{\pi,\operatorname{dist}_{r}(q)}) > 1 - (1-\frac{1}{k!})^{r+1}$ and thus the number of permutations $\pi$ containing $\operatorname{dist}_{r}(q)$ is at least $n!(1 - (1-\frac{1}{k!})^{r+1})$ from which we can deduce that $|\operatorname{Av}_{n}(\operatorname{dist}_{r}(q))| \leq n!((1-\frac{1}{k!})^{r+1}) = ((1-\frac{1}{k!})^{\frac{r+1}{n}})^{n}n! =  d^{n}n!$, for $d = (1-\frac{1}{k!})^{\frac{r+1}{n}}$.
\end{proof}
An analogous fact could be conjectured about the bound from below, which would lead to a proof of Conjecture \ref{conj:SWfracLimit}, since we have Lemma \ref{lemma:SWd}.
\begin{conjecture}\label{conj:SWlowerBound}
For any given classical pattern $q\in S_{k}$ and for every $0<c_{1}<1$, there exists $c>0$, such that if $r = \lfloor c_{1}n\rfloor$ and $n \geq k(r+1)$, then:
\begin{equation}
    |\operatorname{Av}_{n}(\operatorname{dist}_{r}(q))| > c^{n}n!.
\end{equation}
\end{conjecture}

We saw that when $r$ is a positive fraction of $n$, the number of $n$-permutations avoiding the corresponding distant pattern may become huge. Thus it would be reasonable to consider a Stanley--Wilf type conjecture, where $r$ is asymptotically smaller than $\mathcal{O}(n)$, e.g., a function of the kind $n^{c_{2}}$, for $0 < c_{2} < 1$.
\begin{conjecture}\label{conj:SWexp}
For any given classical pattern $q$, there exist $c_{q} > 0$ and $0 <c_{2} < 1$, such that
\begin{equation}
\sqrt[n]{|\operatorname{Av}_{n}(\operatorname{dist}_{r}(q))|}\overset{n\rightarrow \infty}{\underset{r= \lfloor n^{c_{2}} \rfloor}{\longrightarrow}} c_{q}
\end{equation}
\end{conjecture}
If the latter conjecture is true, then one might ask which are the allowable growth rates $c_{q}$ when $c_2$ is a fixed positive constant. Furthermore, an interesting additional question would be to find a function $g(n)$, such that
\begin{equation*}
    \sqrt[n]{|\operatorname{Av}_{n}(\operatorname{dist}_{r}(q))|}\overset{n\rightarrow \infty}{\underset{r= \lfloor \Theta(g(n)) \rfloor}{\longrightarrow}} c, 
\end{equation*}
 for some constant $c>0$, but 
\begin{equation*}
    \sqrt[n]{|\operatorname{Av}_{n}(\operatorname{dist}_{r}(q))|}\overset{n\rightarrow \infty}{\underset{r= \lfloor \Omega(g(n)) \rfloor}{\centernot\longrightarrow}} c, 
\end{equation*}
for any $c>0$.

\section{Open problems and future work} \label{sec:open}
In this section, we list some ideas for possible further investigations, related to the current work on distant patterns:
\begin{itemize}
    \item The approach using the map $g$ in section \ref{sec:len2} can be applied to the set $A'_{n-1} = S_{n-1}\setminus \operatorname{Av}_{n}(1\square 23)\times [n]$ with the hope to find an equation that will give us a recurrence formula for $|\operatorname{Av}_{n}(1\square 2 \square 3)|$. We have checked that a statement analogous to Theorem \ref{th:almostInjective} holds and that no permutation is the image of $g$ for more than two different elements of $A'_{n-1}$.
    
    \item One can try to prove or refute the following surprising conjecture (see section \ref{sec:polygons} for a discussion):
    
    \begin{conjecture}\label{conj:PutSquare}
Choose one of the 3 places between consecutive letters in the patterns $\{1234,1243,2143\}$ and put a square at that place for each of the three given classical patterns. You will obtain three Wilf-equivalent distant patterns. For example
\begin{equation}
|\operatorname{Av}_{n}(1 \square 234)| = |\operatorname{Av}_{n}(1 \square 243)| = |\operatorname{Av}_{n}(2 \square 143)|    
\end{equation}
\end{conjecture}

We should note that a similar statement does not hold for any two Wilf-equivalent classical patterns, because $|\operatorname{Av}_{n}(4132)| = |\operatorname{Av}_{n}(3142)|$ \cite{stankova}, for all $n>1$, but $|A_{7}(4 \square 132)| = 3592 \neq 3587 = |A_{7}(3 \square 142)|$.     

    \item  Three more conjectures related to the least and most avoided uniform distant patterns can be investigated (see \textbf{section \ref{sec:len3}} for a discussion):
    
    \begin{conjecture} \label{conj:asympt123}
 For every $m \geq 3$ and $r \geq 1$, there exists $n_0 \in \mathbb{N}$ such that for every natural $n > n_{0}$, we have 
 \begin{equation}
 |\operatorname{Av}_{n}(\operatorname{dist}_{r}(12 \cdots m))| \geq |\operatorname{Av}_{n}(\operatorname{dist}_{r}(q))|,  
 \end{equation}
  for any given classical pattern $q$ of size $m$.
 \end{conjecture}
 \begin{conjecture}\label{conj:asympt132}
 For every $m \geq 3$ and $r \geq 1$, there exists $n_0 \in \mathbb{N}$ such that for every natural $n > n_{0}$, we have
 \begin{equation}
     |\operatorname{Av}_{n}(\operatorname{dist}_{r}(12 \cdots (m-2)m(m-1)))| \leq |\operatorname{Av}_{n}(\operatorname{dist}_{r}(q))|,
 \end{equation} 
 for any given classical pattern $q$ of size $m$.
 \end{conjecture}
 A weaker version of these two conjectures would be the one below and a suitable injection establishing the fact is desired.
 \begin{conjecture}\label{conj:asympt123132}
 For every $m \geq 3$ and $r \geq 1$, there exists $n_0 \in \mathbb{N}$ such that for every natural $n > n_{0}$:
 \begin{equation}
     |\operatorname{Av}_{n}(\operatorname{dist}_{r}(12 \cdots m))| \geq |\operatorname{Av}_{n}(\operatorname{dist}_{r}(12 \cdots (m-2)m(m-1)))|.
 \end{equation}
 \end{conjecture}
    
    \item Section \ref{sec:conj} gives interpretations in terms of distant patterns to six out of the ten conjectures of Kuszmaul, as well as solutions to two of these six conjectures. One may try to give a solution for the remaining four in a similar fashion, using the listed interpretations. 
    \item The PermPAL database \cite{PermPAL} currently contains about 17000 permutation pattern avoidance results and one may try to find equivalent interpretations in terms of distant patterns, as well as new combinatorial proofs using those interpretations (similar to the two proofs in section \ref{sec:conj}). One may also use the database \cite{PPdatabase}.
    \item Section \ref{sec:SW} mentions a few  unresolved conjectures, namely Conjecture \ref{conj:SWfracLimit}, \ref{conj:SWlowerBound} and \ref{conj:SWexp}, all related to Stanley--Wilf type results.
    \item One may investigate the case of \emph{bivincular distant patterns}, i.e., when we have constraints for the values of the letters in an occurrence of a distant pattern in addition to  gap size constraints.
    \item Several statistics over permutations avoiding certain distant patterns might be considered and perhaps some facts related to their distribution could be established. We were not able to find much previous work related to the topic.
\end{itemize}
\section*{Acknowledgement} I am grateful to my advisor Bridget Tenner for the helpful comments and the ideas related to this work, as well as to Toufik Mansour for providing me the thesis \cite{thesis}.

\end{document}